\documentclass[11pt,a4paper,reqno]{amsart}
\usepackage{geometry}               
\usepackage{amsmath}
\usepackage{amsthm}
\usepackage{amstext,amsopn,amssymb}
\usepackage{amsxtra}
\usepackage{amscd,amsrefs}

\usepackage{dsfont}
\usepackage{tikz}
\usepackage{color}
\usepackage{pdfsync}
\usepackage{umoline}

\usepackage{kotex}
 \usepackage[unicode]{hyperref}

\newtheorem{thm}{Theorem}[section]

\newtheorem{lem}[thm]{Lemma}
\newtheorem{prop}[thm]{Proposition}
\theoremstyle{definition}
\newtheorem{defn}[thm]{Definition}
\theoremstyle{remark}

\numberwithin{equation}{section}

\newcommand{\texteq}[1]{\begin{equation}\text{\parbox{0.85\textwidth}{#1}}\end{equation}}
\newcommand{\set}[1]{\left\{#1\right\}}
\DeclareMathOperator{\supp}{supp}

\begin{document}
\date{\today}%

\title{An endpoint estimate of the bilinear paraboloid restriction operator}

\author[J. Lee] {Jungjin Lee}

\address{Department of Mathematical Sciences, School of Natural Science, Ulsan National Institute of Science and Technology, UNIST-gil 50, Ulsan 44919, Republic of Korea}
\email{jungjinlee@unist.ac.kr}

\thanks{This research was supported by Basic Science Research Program through the National Research Foundation of Korea(NRF) funded by the Ministry of Education (NRF-2020R1I1A1A01062209)
}



\begin{abstract}
In Fourier restriction problems, a cone and a paraboloid are model surfaces. The sharp bilinear cone restriction estimate was first shown by Wolff, and later the endpoint was obtained by Tao. For a paraboloid, the sharp $L^2$ bilinear restriction estimate was obtained by Tao, but the endpoint was remained open. In this paper we prove the endpoint $L^2$ bilinear restriction estimate for a paraboloid.
\end{abstract}

\maketitle

\section{Introduction}
\noindent Fix $n \ge 2$, let $\Sigma$ be a hypersurface 
defined by
\(
  \Sigma = \{(\xi, s(\xi)): \xi \in \mathbb R^n\}
\). Then the (adjoint) Fourier restriction operator 
$\mathcal R_\Sigma f$ for the hypersurface $\Sigma$ can be 
defined by
\[
  \mathcal R_\Sigma f(x,t) := \int e^{2\pi i (x \cdot 
  \xi + t s(\xi))} f(\xi) a(\xi)d\xi,
\]
where $a(\xi)$ is a smooth cut-off function.

The (adjoint) restriction estimate 
$\mathds R^*_\Sigma (p,q)$ for $1 \le p, q \le \infty$ 
is of the form
\begin{equation} \label{rest}
  \| \mathcal R_\Sigma f \|_q \le C_{p,q,\Sigma} \| f\|_p,
\end{equation}
and the restriction problem is to determine $(p,q)$ for
which the estimate $\mathds R^*_\Sigma (p,q)$ holds. 
There are two representative model hypersurfaces. One is 
a cone $\Sigma_{cone}=\{ (\xi, |\xi|): \xi \in \mathbb 
R^n\}$, and the other is a paraboloid 
$\Sigma_{parab}=\{ (\xi, -\frac{|\xi|^2}{2}): \xi \in
\mathbb R^n\}$. 
%
For these two surfaces the restriction operators $\mathcal
R_{\Sigma_{cone}}f$ and $\mathcal R_{\Sigma_{parab}}f$ are 
related to other problems such 
as the Bochner-Riesz conjecture, Kakeya conjecture and 
Sogge's local smoothing conjecture, 
see \cites{tvv,l0,tv2,mss,w2,b}. 
Moreover, they are also connected to the wave and 
Schr\"odinger equations because $\mathcal 
R_{\Sigma_{cone}} \hat f$ and $\mathcal
R_{\Sigma_{parab}}\hat f$ are the solutions to the free wave
equation \( u_{tt} - \Delta u =0 \) and the free Schr\"odinger 
equation \( 4\pi i \partial_t u - \Delta u = 0 \), 
respectively, see \cites{tv2,t1,l3}.

The bilinear restriction estimate $\mathds 
R^*_{\Sigma_1,\Sigma_2}(p \times p, q)$, $1 \le p, q \le 
\infty$ is of the form
\begin{equation*}
  \| \mathcal R_{\Sigma_1} f \mathcal R_{\Sigma_2} g\|_q
  \le C_{p,q,\Sigma_1,\Sigma_2}
  \| f\|_p \| g\|_p,
\end{equation*}
where $\Sigma_1, \Sigma_2$ are two compact subsets of 
$\Sigma$ such that the set
of unit normal vectors of $\Sigma_1$ are separated by
a non-zero distance from the set of unit normal vectors of 
$\Sigma_2$. 
This bilinear restriction estimate 
$\mathds R^*_{\Sigma_1,\Sigma_2}(p \times p, q/2)$ 
was used to improve the linear restriction estimate 
$\mathds R^*_\Sigma(p,q)$.
(The restriction estimate have been improved further by Bourgain--Guth \cite{bg}, Guth \cites{g1,g2}, Wang \cite{wa}, Hickman--Rogers \cite{hr}.)
In addition, as the relation between the Stein-Tomas restriction theorem and the Strichartz estimate, the bilinear restriction estimates $\mathds R^*_{\Sigma_1,\Sigma_2}(2 \times 2, q)$ lead to the corresponding bilinear estimates applied to null form estimates for the relevant dispersive equations, see \cites{lv,tv2,t1,fk,km1,km2}.

The $L^2$ bilinear restriction estimate is based on the 
argument of Wolff \cite{w} for a cone. His 
arguments are roughly composed of two steps. 
One is to use induction to avoid some critical case of 
the Kakeya set. The other is to deal with the remaining relaxed 
Kakeya set by utilizing some geometrical observation as follows:
\texteq{ \label{star}
  The union $\Lambda(x_0)$ of all lines passing through a 
  point $x_0$ and of direction normal to $
  \Sigma_{cone}$ becomes a    cone. 
}
We can see that if a line $\ell$ 
passes through $\Lambda(x_0)$, then $\ell \cap 
\Lambda(x_0)$ has at most $O(1)$ points, which is the key
to obtain the sharp bilinear cone restriction. 

However, in a paraboloid the analogous property does not hold. Specifically, the union 
$\Lambda(x_0)$ of all lines passing through 
$x_0$ and of direction normal to $\Sigma_{parab}$ does 
not contained in a hypersurface. The reason is that
while the cone has one vanishing principle curvature, 
the paraboloid has non-vanishing Gaussian curvature. 
Thus, $\ell \cap \Lambda(x_0)$ may have infinitely 
many points. Because of this difference, Wolff's argument  
cannot be directly applied to the paraboloid case.

This difficulty was resolved by Tao \cite{t2} who used a kind of orthogonality due to the non-vanishing curvature. 
Such an orthogonality was first observed in the proof of the 2-dimensional restriction theorem by Fefferman and C\'ordoba. 
By combining Wolff's arguments with the orthogonality Tao obtained the 
sharp bilinear restriction estimate $\mathds 
R^*_{\Sigma_1,\Sigma_2}(2 \times 2, p)$, 
$p>\frac{n+3}{n+1}$ for a paraboloid.

It is natural to ask whether the endpoint bilinear estimate $\mathds 
R^*_{\Sigma_1,\Sigma_2}(2 \times 2, \frac{n+3}{n+1})$ is valid or not. Since the Kakeya example does not work in the endpoint $L^2$ bilinear restriction estimate, we can expect it.  Tao \cite{t1} obtained the endpoint bilinear cone restriction estimate by exploring energy concentrations and the geometric observation \eqref{star}. If one makes an attempt to prove the endpoint bilinear restriction estimate for a paraboloid, it is reasonable, first of all, to apply Tao's arguments, but the geometric observation \eqref{star} does not hold for the paraboloid restriction operator. However, it still seems to have the $L^2$ bilinear paraboloid restriction estimate because the Kakeya example does not work.

In this paper we will prove the endpoint estimate $\mathds R^*_{\Sigma_1,\Sigma_2}(2 \times 2,\frac{n+3}{n+1})$ for a paraboloid. 
To state more explicitly, 
let $\mathit \Sigma = \mathit \Sigma_{parab}$ and
\begin{align*}
  \mathit \Sigma_j &= \{ (\xi,\tau) \in \mathit \Sigma: 1 
  < |\xi| < 2,\
  \angle(\xi,(-1)^{j-1}e_1) < \pi/8\},\qquad j=1,2, 
\end{align*}
where $e_1 \in \mathbb R^n$ is a standard unit
vector. We define the operator $\mathcal U_j$ 
by $\mathcal U_j f = \mathcal R_{\Sigma_j} \hat f$ for $j=1,2$;
\begin{equation} \label{eqn:Uf}
  \mathcal U_jf(x,t) := \int e^{2 \pi i (x \cdot \xi -
  \frac{1}{2}t|\xi|^2)}
  \widehat f(\xi) a_j (\xi) d\xi,
\end{equation}
where  $a_j$ is a smooth function which is equal to $1$ on
\[
  \Xi_j=\{\xi \in \mathbb R^n : (\xi, \tau) \in \mathit 
  \Sigma_j \}
\]
and supported in 
\[
  \tilde \Xi_j = \{\xi: 1/2 \le |\xi| \le 4,\ \angle(\xi, 
  (-1)^{j-1} e_1) \le \pi/4 \}.
\]

\begin{thm} \label{thm:main}
For $\frac{n+3}{n+1} \le p \le 2$, the estimate
  \begin{equation} \label{est:main}
    \| \mathcal U_1f_1 \mathcal U_2 f_2 \|_{L^p(\mathbb R^n \times \mathbb R)} \le C_p
    \|f_1\|_{L^2(\Xi_1)}\|f_2\|_{L^2(\Xi_2)}
  \end{equation}
holds for all $f_1 \in L^2(\Xi_1)$ and $f_2 \in L^2(\Xi_2)$.
\end{thm}

To obtain the endpoint we basically follow the arguments in \cite{t1}. We first reprove the sharp bilinear restriction estimate for a paraboloid without dyadic pigeonholing. Next we use an energy concentration argument.
But, as mentioned above, we cannot directly apply Tao's endpoint argument to the paraboloid problem because of the lack of \eqref{star}.
To get around this we will devise a new energy concentration argument where the dispersiveness of $\mathcal U_j f$ and the transversality between $\Sigma_1$ and $\Sigma_2$ play a crucial role instead of the geometric observation \eqref{star}.

The Fourier restriction operator can be generalized
to some oscillatory integral operator. The cone 
restriction operator is generalized to the oscillatory 
integrals satisfying the cinematic curvature condition, and 
similarly the paraboloid one  
generalized to the oscillatory integrals with 
the Carleson--Sjölin condition, see \cites{s,st,l2, l1,v}. 
It was shown by S. Lee \cite{l2} that these
two classes of oscillatory integral operators also satisfy 
the estimate \eqref{est:main} for $p > \frac{n+3}{n+1}$, 
provided a suitable transversality condition is assumed, see
\cites{l2,l1,ljj,v}. For the oscillatory integrals with the cinematic curvature
condition, the endpoint bilinear estimate 
was obtained by the author \cite{ljj}. It is likely 
that the oscillatory integral operators with the Carleson--Sjölin condition have the endpoint bilinear estimate, too.


\medskip
\noindent \underline{\textsc{Notation}}. 
Let $N >1$ be a large integer depending only on the dimension $n$, which is used as a large exponent of the error terms. Let $C_0$ be an integer much larger than $N$. 

We use $C$ to denote various large numbers which vary each time. It may depend on $N$ but not depend on $C_0$.
The notation $A \lesssim B$ or $A=O(B)$ implies 
$A \le CB$. If $A \lesssim B$ and $B \lesssim A$ we write $A \sim B$. 


When $\phi(x,t)$ is a space-time function, let 
$\phi(t)$ denote the spatial function $\phi(t)(x) = \phi(x,t)$.
The hat $~\widehat{}~$ notation is used for the Fourier transform
\[
  \widehat f(\xi) = \int_{\mathbb R^n} e^{-2\pi i 
  x\cdot\xi}   f(x) dx.
\]

Let a spacetime cube $Q(x_Q,t_Q;R_Q)$ be an $(n+1)$-dimensional cube in $\mathbb R^{n+1}$ of side-length $R_Q$ 
and centered at $(x_Q,t_Q) \in \mathbb R^n \times \mathbb 
R$ with sides parallel to the axes. Let $D(x_D,t_D;r_D)$ denote an $n$-dimensional disc at time $t_D$ of the form
\[
  D(x_D,t_D;r_D) = \{ (x,t_D) \in \mathbb R^n \times 
  \mathbb R: |x-x_D| \le r_D \}.
\]
For any compact subset $\pi \subset \mathbb R^n$ we define 
a conic set $\Lambda_{\pi}(x_0,t_0)$ with vertex 
$(x_0,t_0) \in \mathbb R^n \times \mathbb R$ by
\begin{equation} \label{def:Cone}
  \Lambda_{\pi}(x_0,t_0) = \set{(x,t)
  \in \mathbb R^n \times \mathbb R  :
  x=x_0 + (t-t_0)w,\ w \in \pi,\ t \in
  \mathbb R},
\end{equation}
and let $\Lambda_{\pi}(x_0,t_0;r)$ be the
$r$-neighborhood of $\Lambda_{\pi}(x_0,t_0)$. 
Briefly we write 
\begin{equation} \label{eqn:Cone2}
  \Lambda_j(x_0,t_0) = \Lambda_{\tilde\Xi_j}(x_0,t_0)
\end{equation} 
for $j=1,2$, and
\[
  \Lambda_{\cup}(x_0,t_0) = \Lambda_1(x_0,t_0) \cup \Lambda_2(x_0,t_0).
\]
If $Q=Q(x_Q,t_Q;r_Q)$ and $c>0$ the $cQ$ is defined by $Q(x_Q,t_Q;cr_Q)$. 
Similarly, $cD$ is defined.

Let $\eta$ be a nonnegative Schwartz function on 
$\mathbb R^n$ with $\int \eta = 1$ and whose Fourier 
transform is supported on the unit disc. 
By the Poisson summation formula we may have
\begin{equation} \label{etaP}
  \sum_{k \in \mathbb Z^n} \eta(x-k) = 1.
\end{equation}
We define $\eta_r$ for $r>0$ by
\begin{equation} \label{fourier_bump}
\eta_r(x) = r^{-n}\eta(x/r).
\end{equation}

\section{Proof of Theorem \ref{thm:main}} \label{sec:EC}
\noindent In this section we state some
propositions and using them we prove Theorem \ref{thm:main}. 
The proof of propositions are given in next sections.

We denote by
\[
  \phi_j(x,t) = \mathcal U_j f_j(x,t),
\]
and define the \textit{energy} $E(\phi_j)$ by
\begin{equation} \label{def:energy}
  E(\phi_j) = \| \phi_j(t) \|_2^2
\end{equation}
where $t \in \mathbb R$ is arbitrary. It makes 
sense by Plancherel's theorem and 
\begin{equation} \label{FForm}
  \widehat{\phi_j(t)}(\xi) =  e^{-\pi i t|\xi|^2 } 
  \hat f_j(\xi) a_j(\xi).
\end{equation}
Using these notations we rewrite Theorem \ref{thm:main} as follows.
\begin{thm} \label{thm:thm2}
For $\frac{n+3}{n+1} \le p \le 2$, the estimate
\begin{equation} \label{thm2}
  \|\phi_1 \phi_2\|_p \le C_p
  E(\phi_1)^{1/2}E(\phi_2)^{1/2}
\end{equation}
holds for all $\phi_1$ and $\phi_2$ whose Fourier transforms are supported in $\mathit \Sigma_1$ and $\mathit \Sigma_2$ respectively.
\end{thm}
The estimate \eqref{thm2} for $p=2$ is well known. Thus, by interpolation it suffices to prove the theorem for \[
p :=\frac{n+3}{n+1}.\]  
\begin{defn} \label{defn:hypoth}
  For any $R \ge C_0$ we define a constant $\mathcal K(R)$ to be the best 
  constant for which the estimate
  \begin{equation} \label{hypEst}
    \|\phi_1 \phi_2\|_{L^p(Q_R)} \le \mathcal K(R)
    E(\phi_1)^{1/2}E(\phi_2)^{1/2}
  \end{equation}
holds for all spacetime cubes $Q_R$ of sidelength $R$ and 
all $\phi_1$, $\phi_2$ of which Fourier transforms are supported in $\mathit \Sigma_1$ and $\mathit \Sigma_2$ respectively and satisfy 
\begin{equation} \label{mar_cond}
  \mathrm{marg}(\phi_1), \mathrm{marg}(\phi_2)
  \ge 1/100-R^{-1/N},
\end{equation}
where the \textit{margin} $\mathrm{marg}(\phi_j)$ is defined by
\[
  \mathrm{marg}(\phi_j) := \mathrm{dist}(\supp(\widehat 
  \phi_j), \mathit \Sigma \setminus \mathit\Sigma_j).
\] 

\end{defn}
Note that the margin condition can be removed by partitioning both $x$-space and $\xi$-space and some Lorentz transforms, see \cite{t1}. 
By the above definition it suffices to show
\begin{equation} \label{goal}
  \mathcal K(R) \le 2^{CC_0}.
\end{equation} 
We may assume that
\begin{equation} \label{eqn:energy assum}
  E(\phi_1)=E(\phi_2)=1.
\end{equation}
By some trivial estimates it follows 
\begin{equation} \label{trivial}
  \|\phi_1 \phi_2\|_{L^p(Q_R)} \lesssim R^C.
\end{equation}
Thus we see that
\begin{equation} \label{triK}
\mathcal K(R) \lesssim R^C.
\end{equation}
By this estimate we may assume $R \ge 2^{C_0}$.
Let \[
\overline{\mathcal K}(R) := \sup_{2^{C_0} \le R' \le R} \mathcal K(R'). \]
\begin{prop} \label{prop:sharp}
Let $R \ge 2^{C_0}$ and $0 < c \le 2^{-C_0}$. 
Suppose that $\phi_1$, $\phi_2$ are Fourier supported in $\mathit \Sigma_1$ and $\mathit \Sigma_2$ respectively and satisfy the margin condition
\begin{equation} \label{Bmargin}
\mathrm{marg}(\phi_j) \ge 1/100 - 2R^{-1/N}, \qquad j=1,2.
\end{equation}
Then,
\begin{equation} \label{recBil}
\| \phi_1 \phi_2 \|_{L^p(Q_R)} 
\le \Big( (1+Cc) \overline{\mathcal K}(R/C_0) + c^{-C} \Big) E(\phi_1)^{1/2}E(\phi_2)^{1/2},
\end{equation}
for all cubes $Q_R$ of sidelength $R$.
\end{prop}
\noindent We will prove this proposition in Section \ref{sec:nonPP}. It is obtained by refining the proof of \cite{t2}.
Technically the constant $(1+Cc)$ is important to obtain the endpoint. 
Note that the estimate \eqref{recBil} implies 
\(
{\mathcal K}(R) \le (1+Cc)  \overline{\mathcal K}(R/C_0) + c^{-C}.
\)
By iterating this estimate it follows $ \overline{\mathcal K}(R) \le C_\varepsilon R^\varepsilon$ for all $\varepsilon>0$. 

\smallskip

For other propositions we define an energy concentration. We first introduce several constants relevant to the conic sets $\Lambda_1(0)$ and $\Lambda_2(0)$ defined as \eqref{def:Cone}. 
\begin{itemize}
\item
 Let $A_w$ be the maximum angle  between two lines $L_j$ and $L_j'$ passing through the origin and contained in $\Lambda_j(0)$ for $j=1,2$. 
\item Let $A_d$ be the minimum angle between two lines $L_1 \subset \Lambda_1(0)$ and $L_2 \subset \Lambda_2(0)$ passing through the origin.
\item Define the constant $A_*$ by
\begin{equation} \label{A*}
A_* := 4A_wA_d^{-1} + C.
\end{equation}
\end{itemize}
\begin{defn} \label{def:EC}
Let $\phi_1,\phi_2$ be Fourier supported in $\mathit \Sigma_1$ and $\mathit \Sigma_2$ respectively.
For $0 \le \varepsilon < 1$, $r>0$ and $t \in \mathbb R$, 
let $\mathfrak D_{t}^{\varepsilon}(\phi_1,\phi_2)$ be the collection of discs $D=D(x_D,t_D;r_D)$ in $\mathbb R^n \times \{t\}$ with $r_D \ge C_0$ such that 
\begin{equation} \label{DetCen}
\| \phi_1\|_{L^2(D(x_D,t_D;\frac{C_0}{800A_*}))}  \|\phi_2\|_{L^2(D(x_D,t_D;\frac{C_0}{800A_*}))}  \ge \varepsilon  E(\phi_1)^{1/2} E(\phi_1)^{1/2}.
\end{equation}
We define the \textit{energy concentration} $E^{\varepsilon}_{r,\mathring r,t}(\phi_1,\phi_2)$ of $\phi_1$ and $\phi_1$ at time $t$ by
\[
E^{\varepsilon}_{r,\mathring r,t}(\phi_1,\phi_2) = \max\Big( \frac{1}{2}
E(\phi_1)^{1/2} E(\phi_1)^{1/2},  \sup_{\substack{D \in \mathfrak D_{t}^{\varepsilon}(\phi_1,\phi_2)\\: r_D = r}} \sup_{\substack{D_1, D_2 \subset \mathcal N(D)\\:r_{D_1}=r_{D_2}=\mathring r}}\|\phi_1\|_{L^2(D_1)}  \|\phi_2\|_{L^2(D_2)} \Big)
\]
where $r_D$ denotes the radius of $D$ and $\mathcal N(D)$ denotes the $A_*$-dilated disc of $D$. (See Figure \ref{fig:Ecct}.)
\end{defn}
The condition \eqref{DetCen} is a technical thing to handle errors. Since $\phi_1,\phi_2$ are compactly Fourier supported, they have Schwartz tails, and by using the Paley--Wiener theorem we can see that for any proper disc $D$, the $\|\phi_j\|_{L^2(D)}$ is nonzero for $j=1,2$.   

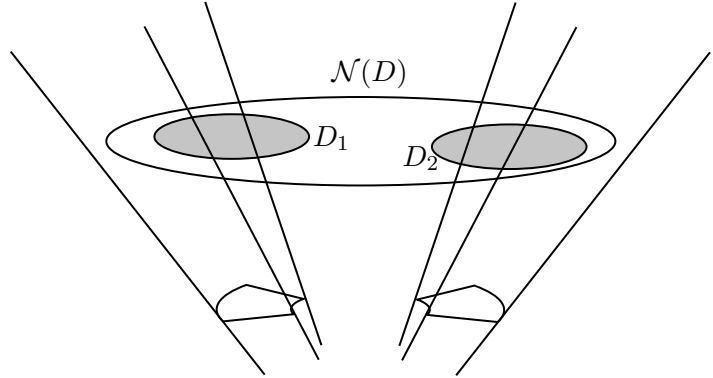
\begin{figure}[htbp]
\begin{center}
\tikzset{every picture/.style={line width=0.75pt}} 
\begin{tikzpicture}[x=0.75pt,y=0.75pt,yscale=-1,xscale=1]

\draw   (123.67,354.08) .. controls (123.67,341.79) and (180.53,331.83) .. (250.67,331.83) .. controls (320.81,331.83) and (377.67,341.79) .. (377.67,354.08) .. controls (377.67,366.37) and (320.81,376.33) .. (250.67,376.33) .. controls (180.53,376.33) and (123.67,366.37) .. (123.67,354.08) -- cycle ;
\draw  [fill={rgb, 255:red, 155; green, 155; blue, 155 }  ,fill opacity=0.58 ] (147.66,351.72) .. controls (147.66,345.52) and (164.94,340.5) .. (186.24,340.5) .. controls (207.55,340.5) and (224.82,345.52) .. (224.82,351.72) .. controls (224.82,357.92) and (207.55,362.94) .. (186.24,362.94) .. controls (164.94,362.94) and (147.66,357.92) .. (147.66,351.72) -- cycle ;
\draw  [fill={rgb, 255:red, 155; green, 155; blue, 155 }  ,fill opacity=0.58 ] (286.11,356.84) .. controls (286.11,350.64) and (303.38,345.62) .. (324.69,345.62) .. controls (346,345.62) and (363.27,350.64) .. (363.27,356.84) .. controls (363.27,363.04) and (346,368.06) .. (324.69,368.06) .. controls (303.38,368.06) and (286.11,363.04) .. (286.11,356.84) -- cycle ;
\draw   (181.96,444.7) .. controls (174.91,438.37) and (179.59,431.41) .. (193.57,426.36) -- (222.06,433.4) .. controls (216.19,435.53) and (214.22,438.44) .. (217.18,441.1) -- cycle ;
\draw    (76,309) -- (202.67,471.5) ;
\draw    (142.67,296.33) -- (229.67,464) ;
\draw    (173,284) -- (231,456.67) ;
\draw   (318.92,445.03) .. controls (325.97,438.7) and (321.29,431.75) .. (307.31,426.69) -- (278.83,433.74) .. controls (284.69,435.86) and (286.66,438.78) .. (283.7,441.44) -- cycle ;
\draw    (424.88,309.33) -- (298.22,471.83) ;
\draw    (358.22,296.67) -- (271.22,464.33) ;
\draw    (327.88,284.33) -- (269.88,457) ;
\draw (234.67,311.23) node [anchor=north west][inner sep=0.75pt]    {$\mathcal{N}( D)$};
\draw (225.33,344.9) node [anchor=north west][inner sep=0.75pt]    {$D_{1}$};
\draw (270,355) node [anchor=north west][inner sep=0.75pt]    {$D_{2}$};
\end{tikzpicture}
\caption{Energy concentration}
\label{fig:Ecct}
\end{center}
\end{figure}

\begin{defn} \label{defn:hypoth2}
Let $R \ge 2^{C_0}$, $r>0$ and $0 \le \varepsilon <1$. We define $\mathcal K_{\varepsilon}(R,r,\mathring r)$ to be the best constant for which the estimate 
\begin{equation} \label{eqn:hypoth2}
\|\phi_1\phi_2\|_{L^p(Q_R)}
\le \mathcal K_{\varepsilon}(R,r, \mathring r) (E(\phi_1)^{1/2} E(\phi_1)^{1/2})^{1/p} E^{\varepsilon}_{r, \mathring r, t_e}(\phi_1, \phi_2)^{1/p'}
\end{equation}
holds for all spacetime cube $Q_R$ of sidelength $R$, all $t_e \in \mathbb R$ and  all $\phi_1, \phi_2$ whose Fourier transforms are supported in $\mathit \Sigma_1$ and $\mathit \Sigma_2$ respectively and satisfy the margin condition \eqref{mar_cond}. 
\end{defn}
%

From the definitions of $\mathcal K(R)$ and $\mathcal K_{\varepsilon}(R,r,\mathring r)$ it immediately follows
\begin{align*} 
\mathcal K(R) &\le  \mathcal K_{\varepsilon}(R,r,\mathring r), \\
\mathcal K_{\varepsilon}(R,r,\mathring r) &\le 2^{1/p'} \mathcal K(R).
\end{align*}
By the dispersive property of $\phi_1,\phi_2$ the above estimates are further improved under certain circumstances. 
\begin{prop} \label{prop:energy concent}
For any $R \ge 2^{C_0}$ and $0 < \varepsilon \ll 1$,
  \begin{equation} \label{imener}
    \mathcal K(R) \le (1-C_0^{-C}) \sup_{\substack{r \ge C_0^{-C}R,  \\ r/100< \mathring r \le (2A_wA_d^{-1}+C_0^{-C})r}}
    \mathcal K_{\varepsilon}(R, r, \mathring r ).
  \end{equation}
\end{prop}
This proposition will be proven in Section \ref{conPP}. In the above estimate, the constant $(1-C_0^{-C})$ is crucial for closing an induction. The supremum condition $r \ge C_0^{-C}R$ is also important, which prevents a loss caused by iteration. Because of this condition we need only $O(1)$ iteration.

\begin{prop} \label{prop:core}
Let $R \ge 2^{C_0}$. If $R^{-N/4} \le \varepsilon \ll 1$, $r \ge C_0^{C}A_*R$ and $r/100 < \mathring r \le (2A_wA_d^{-1}+C_0^{-C})r$ then
\[
  \mathcal K_{\varepsilon}(R,r,\mathring r) \le {(1+Cc)}{\overline{\mathcal K}(R)} + 2^{CC_0}
\]
for any  $0<c\le 2^{-C_0}$.
\end{prop}
\noindent We will prove this in Section \ref{sec:core}.
In the above estimate, the constant $(1+Cc)$ is crucial. Because of the condition $r \ge C_0^{C}R$ in the above proposition, we can combine this with the previous one. To resolve it we use the following recursive estimate.  
\begin{prop} \label{prop:scale}
Let $R \ge 2^{C_0}$. If $0<\varepsilon \ll 1$, $C_0^{-C}R \le r \le C_0^{C}R$ and $\mathring r >0$,
\[
\mathcal K_{\varepsilon}(R, r, \mathring r) \le (1+Cc) \mathcal K_{\varepsilon}(R/C_0,r_{\natural}, \mathring r_{\natural})+c^{-C}
\] 
for any $0 < c \le 2^{-C_0}$, where $r_{\natural} := r(1-Cr^{-1/3N})$ and $\mathring r_{\natural} := (\mathring r)_{\natural}$.
\end{prop}
\noindent We will prove this in Section \ref{sec:endPP}. The above recursive estimate is obtained by modifying the proof of Proposition \ref{prop:sharp}.

\medskip
 
To prove \eqref{goal} we combine the above three propositions.
From Proposition \ref{prop:core} and Proposition \ref{prop:scale} it follows that if $R^{-N/4 } \le \varepsilon \ll 1$, $ r \ge C_0^{-C}R$ and  $\mathring r \le (2A_wA_d^{-1}+C_0^{-C})r$, then
\begin{equation} \label{uni}
\mathcal K_{\varepsilon}(R,r, \mathring r) \le (1+Cc) \overline{\mathcal K}(R) + 2^{CC_0}+c^{-C}.
\end{equation}
Indeed, by Proposition \ref{prop:core} we may assume that $C_0^{-C}R \le r \le C_0^{C}R$.  Let $r_0 :=r$, $r_{j+1}:= (r_j)_{\natural}$ and $\mathring r_{j+1}:= (\mathring r_j)_{\natural}$
for $j=0,1,2,\cdots$. Then if $\mathring r \le (2A_wA_d^{-1}+C_0^{-C})r$ then $\mathring r_{\natural} \le (2A_wA_d^{-1}+C_0^{-C})r_{\natural}$. We take $J$ as the smallest integer such
that $r \ge C_0^{C-J}A_*R$.  Because of the condition $r \ge
C_0^{-C}R$, we have $J=O(1)$. 
From Proposition \ref{prop:scale} it follows that
\[
  \mathcal K_{\varepsilon}(R/C_0^{j},r_j, \mathring r_j) \le (1+Cc) \mathcal K_{\varepsilon}( R/C_0^{j+1},r_{j+1}, \mathring r_{j+1}) +  c^{-C}.
\]
By iteration we have
\[
  \mathcal K_{\varepsilon}(R,r,\mathring r) \le (1+Cc)^J \mathcal K_{\varepsilon}(R/C_0^J, r_J,\mathring r_J) + 2^J c^{-C}.
\]
Since $r \ge C_0^{-C}R \ge C_0^{-C}2^{C_0}$ and $C_0$ is very large, we see that $r_J$ is comparable to $r$.
By Proposition \ref{prop:core}, 
\begin{equation*}
\mathcal K_{\varepsilon}(R,r,\mathring r) \le (1+Cc)^{{J+1}} {\overline{\mathcal K}(R/C_0^J)} + 2^{CC_0}+ 2^{{J+1}} c^{-C}.
\end{equation*}
Since $\overline{\mathcal K}(R/C_0^J) \le \overline{\mathcal K}(R)$, this estimate implies \eqref{uni}.

Combine Proposition \ref{prop:energy concent} with \eqref{uni}. Then,
\[
  \mathcal K(R) \le (1-C_0^{-C})  (1+Cc) \overline{\mathcal K}(R) + 2^{CC_0} + c^{-C}.
\]
If we set $c=2^{-C_0}$ then 
$
  \overline{\mathcal K}(R) \le (1-C_0^{-C}) \overline{\mathcal K}(R) + 2^{CC_0}.
$
By rearranging we obtain $\overline{\mathcal K}(R) \le C_0^{C}2^{CC_0} \le 2^{C'C_0}$, which implies
\eqref{goal}.

\section{Preliminaries for the bilinear restriction estimate}

Let $R \ge 2^{C_0}$, $0< c \le 2^{-C_0}$ and $\kappa>0$ an integer such that
\begin{equation} \label{kapa}
  r:=2^{-\kappa}R \sim R^{1/2}.
\end{equation}

\subsection{A wave packet decomposition} 
Let $L=c^{-2} r \mathbb Z^n$ and $V_j$ be a maximal $r^{-1}$-separated
subset of $\tilde\Xi_j$ for $j = 1,2$.
For each $(x_0,v_j) \in L \times V_j$, we define a \textit{tube}
$T_j=T_j^{(x_0,v_j)}$ with initial position $x_0$ and direction $v_j$
by
\[
  T_j=\big\{(x,t) \in \mathbb R^n \times \mathbb R:
  |t| \le R,~  |(x-x_0)-tv_j| \le  r \big\},
\]
and let $\mathbf T_j$ denote the collection of these tubes. We denote by $x(T_j) = x_0$ the position of $T_j$ and $v(T_j) =v_j$ the velocity of $T_j$.

Now we decompose $\phi_j$ into wave packets essentially supported on tubes $T_j$. 
To partition $\mathbb
R^n$ into cubes of {sidelength} ${c^{-2}r}$, 
we set
\begin{equation} \label{spPar}
  \eta^{x_0}(x):= \eta \Big(\frac{x-x_0}{c^{-2} r}\Big).
\end{equation}
Then,
\begin{equation} \label{spaceDecom}
\sum_{x_0 \in L} \eta^{x_0}(x) = 1.
\end{equation}
To partition $\tilde\Xi_j$, let $B_{v_j}$ be a neighborhood of $v_j \in V_j$ with pairwise disjoint interiors so that
\[
  \tilde\Xi_j = \bigcup_{v_j \in V_j} B_{v_j}.
\]
For each $\omega \in D(0;1/r)$, we define a map
$\Omega_{\omega}:\mathbb R^n \rightarrow \mathbb R^n$ by
$\Omega_{\omega}(v) = v+w$. Let $G$ be the set of these maps, and  
define  $d\Omega$ by 
\[
  \int_G F(\Omega) d\Omega = \frac{1}{|D(0;1/r)|}\int_{D(0;1/r)} F(\Omega_w) dw.
\]
For each $\Omega \in G$ and  $v_j \in V_j$, we define the
multiplier $P_{\Omega, v_j}$ by
\[
  \widehat{P_{\Omega,v_j}f} = \chi_{\Omega(B_{v_j})} \widehat f,
\]
where $\chi$ denotes a characteristic function.
For each $T_j=T_j^{(x_0,v_j)} \in \mathbf T_j$, we define a function $f_{T_j}$
by
\begin{equation} \label{eqn:f_T}
f_{T_j}(x) :=\eta^{x_0}(x) \int_G P_{\Omega,v_j}f_j(x) {d\Omega}.
\end{equation}

We define a \textit{wave packet} $\phi_{T_j}$ as
\begin{equation} \label{phiT}
  \phi_{T_j} := \mathcal U_j  f_{T_j}.
\end{equation}
Then by the linearity of Fourier transform, 
\begin{equation} \label{wavepacket decomposition}
\phi_j (x,t)= \sum_{T_j \in \mathbf T_j} \phi_{T_j}(x,t).
\end{equation}


\begin{lem}[Properties of the wave packets] \label{lem:wavepack}
Suppose that $\phi_j$ has a Fourier support in $\mathit \Sigma_j$ and a margin
\begin{equation} \label{marT}
\mathrm{marg}(\phi_j) \ge Cr^{-1}
\end{equation} 
for $j=1,2$. Let $T_j=T_j^{(x_0,v_j)}$.
We define a constant
$h_{T_j}$ by
\begin{equation} \label{hT}
  h_{T_j} := r^{n/2}
  \mathcal M\Big(\int_G P_{\Omega,v_j}f_j {d\Omega}\Big)(x_0),
\end{equation}
where $\mathcal M$ denotes the Hardy--Littlewood maximal operator.
Then we have the followings.
\begin{itemize}
\item The margin of $\phi_{T_j}$ satisfies
\begin{equation} \label{Tmargin}
  \mathrm{marg}(\phi_{T_j}) \ge \mathrm{marg}(\phi_j)-Cr^{-1}.
\end{equation}

\item
For $(x,t) \in \mathbb R^n \times [-CR,CR]$,
\begin{equation}\label{wavepack}
  |\phi_{T_j}(x,t)| \le C_M c^{-C} r^{-n/2} h_{T_j}
  \left( 1+ \frac{\mathrm{dist}(T_j,(x,t))}{r}\right)^{-M}, \qquad \forall M>0.
\end{equation}
 
\item
\begin{equation}\label{l2sum}
  \Big( \sum_{T_j \in \mathbf T_j} h_{T_j}^2 \Big)^{1/2}
  \lesssim c^{-C}\|f_j\|_2.
\end{equation}
\end{itemize}
\end{lem}
\begin{proof}
Consider \eqref{Tmargin}. From the definition of $f_{T_j}$ we can see that the Fourier support of $f_{T_j}$ is contained in a $O(r^{-1})$-neighborhood of $v_j$. So, the spacetime Fourier transform of $\phi_{T_j}$ is supported in a $O(r^{-1})$-neighborhood of the spacetime Fourier support of $\phi_j$. From this we have \eqref{Tmargin}.

Consider \eqref{wavepack}.  If $\rho$ is a smooth bump function supported on a
$O(1)$-neighborhood of the origin and $\rho^{v_j}(\xi) :=\rho\big(4^{-1}r(\xi-v_j) \big)$ then we may replace $\widehat f_{T_j}$ with $\rho^{v_j} \widehat f_{T_j}$.
By interchanging the integrals we may write
\[
  \phi_{T_j}(x,t) = \int K_{v_j}(x-y,t) f_{T_j}(y) dy,
\]
where
\begin{equation} \label{def:Kv}
K_{v_j}(x,t) = \int e^{2\pi i (x\cdot \xi - \frac{1}{2}t|\xi|^2)}
\rho^{v_j}(\xi) d\xi.
\end{equation}
By integration by parts, if $|t| \lesssim R$ then
\begin{equation} \label{ker_est2}
  |K_{v_j}(x,t)| \le C_M r^{-n} \left( 1+
  \frac{|x-tv_j|}{r} \right)^{-M}, \qquad \forall M>0.
\end{equation}
Indeed, let $\delta=r^{-1}$ and
$\Psi(x,t,\xi)=2\pi(x\cdot(\delta\xi+v_j)-
\frac{1}{2}t|\delta\xi+v_j|^2)$. We rewrite as
\begin{equation} \label{wavepacket_ker}
  K_{v_j}(x,t) = \delta^n \int e^{ i \psi(x,t,\xi)}
  \rho(\xi) d\xi.
\end{equation}
Suppose that $|x-tv_j| \ge
C\delta^{-1}$. Then we have
\(
  |\nabla_{\xi} \Psi(x,t,\xi)| \ge C\delta|x-tv_j|,
\)
since $|\delta^2t\xi| \lesssim 1$. By integration by parts,
\[
  |K_{v_j}(x,t)| \le C_M \delta^n (\delta|x-tv_j|)^{-M}, \qquad \forall M>0.
\]
On the other hand, we have a trivial estimate
\(
  |K_{v_j}(x,t)| \lesssim \delta^n.
\)
By combining these two estimates we have \eqref{ker_est2}.
Thus,
\begin{align*}
  |\phi_{T_j}&(x,t)|\\
  &\le C_M r^{-n} \int \left( 1+
  \frac{|(x-y)-tv_j|}{r} \right)^{-M}
  \eta^{x_0}(y) \bigg|\int_G P_{\Omega,v_j}f_j(y) {d\Omega} \bigg| dy \\
  &\le C_M c^{-C} \left( 1+
  \frac{|(x-x_0)-tv_j|}{r} \right)^{-M}
  \mathcal M\Big(\int_G P_{\Omega,v_j}f_j {d\Omega}\Big)(x_0) \\
  &\le C_M c^{-C} r^{-n/2} h_{T_j}
  \left( 1+ \frac{\mathrm{dist}(T_j,(x,t))}{r}\right)^{-M}.
\end{align*}

Consider \eqref{l2sum}. By the uncertainty principle, if $|x-x_0|
\lesssim c^{-2}r$ then
\[
  \mathcal M\Big(\int P_{\Omega,v_j}f_j {d\Omega}\Big)(x_0) \lesssim
  c^{-C}\mathcal M\Big(\int P_{\Omega,v_j}f_j {d\Omega}\Big)(x).
\]
Thus,
\[
  \sum_{T_j \in \mathbf T_j} h_{T_j}^2
  \lesssim c^{-C} \sum_{v_j \in V_j} \int \Big|
  \mathcal M\Big(\int P_{\Omega,v_j}f_j {d\Omega}\Big)(x) \Big|^2 dx.
\]
By the Hardy--Littlewood maximal theorem and Minkowski's inequality, 
\begin{align*}
  \big(\sum_{T_j \in \mathbf T_j} h_{T_j}^2 \big)^{1/2}
  &\lesssim c^{-C}
  \Big(\sum_{v_j \in V_j}\Big\|\int P_{\Omega,v_j}f_j {d\Omega}
  \Big\|_2^2 \Big)^{1/2}\\
  &\le c^{-C} \int \big(\sum_{v_j \in V_j} \|P_{\Omega,v_j}f_j \|_2^2
  \big)^{1/2}{d\Omega}.
\end{align*}
By Plancherel's theorem and orthogonality,
\[
  \sum_{v_j \in V_j} \|P_{\Omega,v_j}f_j \|_2^2 \lesssim
  \|f_j\|_2^2.
\]
Inserting this into the previous integral we have \eqref{l2sum}.
\end{proof}

\subsection{Estimates on a light conic set.}
We define a kernel $K_{j,t}$ by
\begin{equation*}
  K_{j,t}(x) = \int e^{2\pi i (x \cdot \xi - 
  \frac{1}{2}t|\xi|^2)} a_j(\xi) d\xi.
\end{equation*}
Then $\phi_j$ is written as
\begin{equation} \label{kerF}
  \phi_j(x,t)=\mathcal U_jf(x,t) = K_{j,t} \ast f(x).
\end{equation}
\begin{lem} \label{lem:ker_est}
Let $\Lambda_{j,t}=\{x \in \mathbb R^n: (x,t) \in
\Lambda_j(0)\}$ where $\Lambda_j(0)$ is defined as in \eqref{eqn:Cone2}. Then,
\begin{equation} \label{ker_est}
\big| K_{j,t}(x) \big| 
\le C_M
(1+ \mathrm{dist}(\Lambda_{j,t},x))^{-M}, \qquad \forall M>0.
\end{equation}
\end{lem}
\begin{proof}
If $\mathrm{dist}(\Lambda_{j,t},x) = 0$, by a trivial 
estimate we have
\begin{equation*} \label{ker1}
  \big| K_{j,t}(x) \big| \lesssim 1.
\end{equation*}
Suppose that $\mathrm{dist}(\Lambda_{j,t},x) > 0$. The
$\xi$-derivative of the phase $x\cdot \xi - 
\frac{1}{2}t|\xi|^2$ has 
\[
  \Big|\nabla_{\xi} \Big(x\cdot \xi - \frac{1}{2}t|\xi|^2 
  \Big)\Big| \ge \mathrm{dist}(\Lambda_{j,t},x)
\]
for all $\xi \in \tilde \Xi_j$.
So, using integration by parts we obtain
\begin{equation*}
  \big| K_{j,t}(x) \big| \le C_M
  \mathrm{dist}(\Lambda_{j,t},x)
  ^{-M}, \qquad \forall M>0.
\end{equation*}
Thus we have \eqref{ker_est}.
\end{proof}

\begin{lem} \label{englem}
Let $Q$ be a cube of sidelength $R^C$. Let $\epsilon>0$ and $j,k=1,2$ but $j \neq k$. Let $\pi_k \subset \Xi_k$ be a hypersurface in $\mathbb R^n$. Suppose that there is $0 \le \varepsilon_0 < 1$ such that for any $v \in \Xi_j$ and any $w, w' \in \pi_k$ with $w \neq w'$,
\begin{equation} \label{sepa}
  \Big|\Big\langle \frac{v-w}{|v-w|}, \frac{w'-w}{|w'-w|}
  \Big\rangle \Big|\le \varepsilon_0.
\end{equation}
Then for any $r \gtrsim  R^\epsilon$,
\begin{equation*}
  \| \phi_j \|_{L^2(Q \cap \Lambda_{\pi_k}(z_0;r))}
  \lesssim r^{1/2} E(\phi_j)^{1/2}.
\end{equation*}
\end{lem}
%
%
\begin{proof}
It suffices to show
\begin{equation*}
  \int \big\| \chi_{\Lambda^Q_{\pi_k}(z_0;r)}(t) [\mathcal U_j(t)] f
  \big\|_{L^2(\mathbb R^n)}^2 dt
  \lesssim r \|f\|_{L^2(\mathbb R^n)}^2,
\end{equation*}
where $\Lambda_{\pi_k}^Q = \Lambda_{\pi_k} \cap Q$.
By duality this is equivalent to
\begin{equation} \label{dualesti}
  \left\| \int
  \mathcal U_j(t)^*
  \big( \chi_{\Lambda^Q_{\pi_k}(z_0;r)}(t) g(t) \big)dt
  \right\|_{L^2(\mathbb R^n)}^2 \lesssim 
  r \|g\|_{L^2(\mathbb  R^{n+1})}^2.
\end{equation}
The left side of the above estimate is written as
\begin{equation*} \label{de1}
\iint \left\langle [\mathcal{U}_j(s)
\mathcal{U}_j(t)^{\ast}]  \big(
\chi_{\Lambda^Q_{\pi_k}(z_0;r)}(t) g(t) \big)
,~\chi_{\Lambda^Q_{\pi_k}(z_0;r)}(s)g(s) \right\rangle dtds. 
\end{equation*}
Let
\[
G_k(x,t) := \chi_{\Lambda^Q_{\pi_k}(z_0;r)}(x,t) g(x,t)
\]
and 
\[
\mathcal I_jf(s,t) := \left\langle [\mathcal{U}_j(s )
\mathcal{U}_j(t)^{\ast}] f(t),f(s) \right\rangle. 
\]
The previous integral is divided into two parts
\[
\iint \mathcal I_jG_{k}(s,t) dtds= \iint_{|s-t| \le C r} \mathcal I_jG_{k}(s,t) dtds + \iint_{|s-t| \ge C r} \mathcal I_jG_{k}(s,t) dtds.
\]
To show \eqref{dualesti} it suffices to prove
\begin{align} \label{errbb}
  \bigg| \iint_{|s-t| \ge C r} 
  \mathcal I_jG_{k}(s,t)
  dtds \bigg| 
  &\lesssim r^{-N}\|g\|_2^2 \\
\intertext{and}
\label{esspaE}
  \bigg| \iint_{|s-t| \le Cr} 
  \mathcal I_jG_{k}(s,t) 
  dtds \bigg| 
  &\lesssim r \|g\|_2^2.
\end{align} 

Consider \eqref{errbb}.
If we set
\[
  K_{j,t}(x) = \int e^{2\pi i (x \cdot \xi - \frac{1}{2}t|\xi|^2)}
  a_j^2(\xi) d\xi
\]
then
\[
  [\mathcal{U}_j(s) \mathcal{U}_j(t)^{\ast}] f(x)
  = \int K_{j,s-t}(x-y) f(y) dy.
\]
We rewrite $\mathcal I_j G_k$ as
\begin{equation*}
  \mathcal I_j G_k= \iint K_{j,s-t}(x-y) 
  \chi_{\Lambda^Q_{\pi_k}(z_0;r)}(y,t)
  \chi_{\Lambda^Q_{\pi_k}(z_0;r)}(x,s) g(y,t) 
  \overline{g(x,s)}
  dxdy.
\end{equation*}
We divide 
\[
  K_{j,t}= \chi_{\Lambda_{j,t}(0;r)} K_{j,t} + (1-\chi_{\Lambda_{j,t}(0;r)}) K_{j,t}.
\]
From Lemma \ref{lem:ker_est} it follows that
\[
  (1-\chi_{\Lambda_{j,s-t}(0;r)}(x))|K_{j,s-t}(x)| 
  \lesssim r^{-M}, \qquad \forall M>0.
\]
Using this we have
\begin{align*}
  &\bigg| \iiiint ((1-\chi_{\Lambda_{j,s-t}(0;r)})K_{j,s-t})(x-y) \\
  &\qquad\qquad  \times \chi_{\Lambda^Q_{\pi_k}(z_0;r)}(y,t)
  \chi_{\Lambda^Q_{\pi_k}(z_0;r)}(x,s) 
  g(y,t)  \overline{g(x,s)}  dxdydtds \bigg|  \\
  &\qquad\qquad\qquad 
  \lesssim r^{-M} \|\chi_{\Lambda^Q_{\pi_k}(z_0;r)} g\|_1^2 
  \\
  &\qquad\qquad\qquad  
  \lesssim r^{-M} R^{C}\|g\|_2^2   \\
  &\qquad\qquad\qquad 
  \lesssim r^{-M}\|g\|_2^2,  \qquad \forall M>0,
\end{align*}
where the last line follows from $r \gtrsim R^{\epsilon}$.
Now, to show \eqref{errbb} it suffices to show
\begin{equation*} 
\begin{split}
  &\iiiint_{|s-t| \ge C r} \chi_{\Lambda_{j,s-t}(0;r)}(x-y) K_{j,s-t}(x-y) \\
  &\qquad\qquad \times \chi_{\Lambda^Q_{\pi_k}(z_0;r)}(y,t)
  \chi_{\Lambda^Q_{\pi_k}(z_0;r)}(x,s) 
  g(y,t)  \overline{g(x,s)}  dxdydtds = 0.
\end{split}
\end{equation*}
It is enough to show that for $s,t \in \mathbb R$ with $|s-t| \ge Cr$, the equation
\begin{equation} \label{crosK} 
  \chi_{\Lambda_{j,s-t}(0;r)}(x-y) 
  \chi_{\Lambda^Q_{\pi_k}(z_0;r)}(y,t)
  \chi_{\Lambda^Q_{\pi_k}(z_0;r)}(x,s)
\end{equation}
vanishes.  Consider the contrapositive statement that if $\eqref{crosK}$ is nonzero then one has $|s-t| \lesssim r$. 
Suppose that $\eqref{crosK}$ is nonzero. Then, from the characteristic function $\chi_{\Lambda_{j,s-t}(0;r)}(x-y)$ we can restrict ourselves to the case
\begin{equation}\label{r1}
    x-y = (s-t)v + {O}(r)
\end{equation}
for some $v \in \Xi_j$. 
On the other hands,
from 
$\chi_{\Lambda_{\pi_k}(z_0;r)}(y,t)$ and
$\chi_{\Lambda_{\pi_k}(z_0;r)}(x,s)$ we also have
\begin{equation} \label{r2}
  \begin{split}
    x - x_0 &= (s-t_0)w+{O}(r) \\
    y - x_0 &= (t-t_0)w'+{O}(r)
  \end{split}
\end{equation}
for some $w,w'  \in \pi_k$. 
By combining \eqref{r1} and \eqref{r2} it follows that
\begin{equation} \label{r3}
  (s-t)(v-w) + (t-t_0)(w'-w) = O(r).
\end{equation}
If $w = w'$, then we have $|s-t| \lesssim r$. Otherwise, from
\eqref{sepa} we can see that there exists a unit vector $u$ such
that $(v-w) \cdot u \neq 0$ but $(w'-w)\cdot u = 0$. By taking inner
product with such $u$ for \eqref{r3}, we have $|s-t| \lesssim r$.

\medskip
Consider \eqref{esspaE}.
By the Cauchy--Schwarz inequality and Plancherel's theorem it follows that
\[
|\mathcal I_j f(s,t)| \lesssim \|f(s)\|_2 \|f(t)\|_2.
\]
By this and the Hardy--Littlewood--Sobolev inequality,
\begin{align*}
  \bigg| \iint_{|s-t| \lesssim r} 
  \mathcal I_jG_{k}(s,t) 
  dtds \bigg| 
&\lesssim \iint_{|s-t| \lesssim r} 
  \| g(s) \|_{L^2(\mathbb R^n)}
  \| g(t) \|_{L^2(\mathbb R^n)} dtds \\
&\lesssim r\|g\|_2^2.
\end{align*}
Thus we have \eqref{esspaE}.
\end{proof}

\subsection{A basic bilinear restriction estimate}
Let $\psi$ be a nonnegative Schwartz function on $\mathbb R^{n+1}$ with $\int \psi = 1$ such that $\widehat \psi$ is supported in the unit ball and 
\begin{equation} \label{stDec}
\sum_{k \in \mathbb Z^{n+1}} \psi^4(z-k) =1
\end{equation}
where $z = (x,t) \in \mathbb R^{n+1}$.
If $q$ is a cube of sidelength $r_q$ with center $z_q$, we define
\[
\psi_q(z) := \psi(r_q^{-1} (z-z_q)).
\]
For convenience, we use the notations $\int_{\psi} \cdot $ and $\| \cdot \|_{L^2(\psi)}$ to denote
\[
\int_{\psi} f := \int f \psi
\qquad \text{and} \qquad
\| f \|_{L^2(\psi)} := \| f \psi \|_{2}.
\]

\begin{lem} \label{lem:simBil}
Let $q$ be a cube of side-length $R^{1/2}$. 
Suppose that $\phi_{T_1}$ and $\phi_{T_2}$ are wave packets defined as \eqref{phiT}.
Then
\begin{equation} \label{simpleBilEst}
\|\phi_{T_1} \phi_{T_2} \|_{L^2({\psi_q^2})} \lesssim R^{-\frac{n+1}{4}}\|\phi_{T_1}\|_{L^2({\psi_q})} \|\phi_{T_2}\|_{L^2({\psi_q})}.
\end{equation}
\end{lem}
\begin{proof}
By Plancherel's theorem  the estimate \eqref{simpleBilEst} is equivalent to
\begin{equation*}
\| \widehat{\psi_q \phi_{T_1}} \ast \widehat{\psi_q \phi_{T_2}} \|_2 \lesssim R^{-\frac{n+1}{4}}\| \widehat{\psi_q \phi_{T_1}}\|_2 \| \widehat{\psi_q \phi_{T_2}}\|_2.
\end{equation*}
By interpolation it suffices to show the following two estimates:
\begin{align*}
\| \widehat{\psi_q \phi_{T_1}} \ast \widehat{\psi_q \phi_{T_2}} \|_1 &\le \| \widehat{\psi_q \phi_{T_1}}\|_1 \| \widehat{\psi_q \phi_{T_2}}\|_1, \\
\| \widehat{\psi_q \phi_{T_1}} \ast \widehat{\psi_q \phi_{T_2}} \|_\infty &\lesssim R^{-\frac{n+1}{2}}\| \widehat{\psi_q \phi_{T_1}}\|_\infty \| \widehat{\psi_q \phi_{T_2}}\|_\infty.
\end{align*}

By Young's inequality the first one is easily obtained.
Consider the second one. Observe that the Fourier support of $\psi_q \phi_{T_j}$ is contained in a ball $B(v(T_j),-{|v(T_j)|^2}/{2};C R^{-1/2})$.
Let $\chi_j := \chi_{B(v(T_j),-{|v(T_j)|^2}/{2};CR^{-1/2})}$ be a characteristic function. Then, 
\[
\| \widehat{\psi_q \phi_{T_1}} \ast \widehat{\psi_q \phi_{T_2}} \|_\infty \lesssim \| \chi_{1} \ast \chi_{2}\|_{\infty}\| \widehat{\psi_q \phi_{T_1}}\|_\infty \| \widehat{\psi_q \phi_{T_2}}\|_\infty.
\]
Simple computation gives
$\| \chi_{1} \ast \chi_{2}\|_{\infty} \lesssim R^{-\frac{n+1}{2}}$. Thus we have the second estimate. 
\end{proof}

\section{Refining the proof of the sharp bilinear restriction estimate}

In this section we refine the proof of the sharp bilinear paraboloid restriction estimate due to Tao \cite{t2}.

\subsection{Decomposition for bilinear estimates} 
Let $Q$ be the cube of sidelength $CR$ and centered at the origin. We decompose $Q$ into subcubes $\Delta$ of sidelength $2^{-C_0}CR$. For any integer $l$, we define $\mathcal Q_{l}(Q)$ to be the collection of subcubes of sidelength $2^{-l}CR$ which are obtained by dividing the sides of $Q$.


 Let 
\[ 
\Psi_{D(x_D;r_D)}(x) := \left( 1+ \frac{|x-x_D|}{r_D} \right)^{-N^{100}}.
\]
and for each tube $T_j = T_j^{(x_0,v_j)}$, 
\begin{equation} \label{def:X_T} 
\Psi_{T_j}(x,t) := \Psi_{D(x_0-tv_j;r)}(x).
\end{equation}
For each $2^{-C_0}CR$-cube $\Delta \in \mathcal{Q}_{C_0}(Q)$ and each $T_j \in \mathbf T_j$, we define  
\begin{equation}\label{eqn:weight}
  m_{\Delta,T_j} := \sum_{\substack{q \in \mathcal Q_\kappa(Q)  \\: q \subset \Delta}} \sum_{T_i \in \mathbf T_i} \| \phi_{T_i} \Psi_{T_j}
  \|_{L^2(\psi_q)}^2 + R^{-10n}E(\phi_i)
\end{equation}
and
\begin{equation} \label{WT}
  \quad m_{T_j}:= \sum_{{q \in \mathcal Q_\kappa(Q) }} \sum_{T_i \in \mathbf T_i} \|\phi_{T_i} \Psi_{T_j} \|_{L^2(\psi_{q})}^2  + R^{-10n}2^{(n+1)C_0}E(\phi_i)
\end{equation}
for $i \neq j$. Then,
\begin{equation} \label{weighT}
\sum_{\Delta \in \mathcal{Q}_{C_0}(Q)} \frac{m_{\Delta,T_j}}{m_{T_j}} = 1.
\end{equation}

We now define $\Phi_j^{(\Delta)}$ for each $\Delta \in \mathcal{Q}_{C_0}(Q)$ by
\begin{equation} \label{bush function}
  \Phi_j^{(\Delta)}(z) :=
  \sum_{T_j \in \mathbf{T}_j} \frac{m_{\Delta,T_j}}{m_{T_j}}
  \phi_{T_j}(z).
\end{equation}
Then,
\begin{equation} \label{bush decomp}
\phi_j(z) = \sum_{\Delta \in \mathcal{Q}_{C_0}(Q)} \Phi_j^{(\Delta)}(z).
\end{equation}
We define a function $[\Phi_j]$ by
\begin{equation} \label{core}
[\Phi_j](z) := \sum_{\Delta \in \mathcal{Q}_{C_0}(Q)} \Phi_j^{(\Delta)}(z)
\chi_{\Delta}(z).
\end{equation}



%

The main proposition of this section is as follows.
\begin{prop} \label{prop:tao}
Let $R \ge 2^{C_0}$ and $0< c \le 2^{-C_0}$. For any cube $Q$ we define a set $X(Q)$ by
\begin{equation*} 
  X(Q):= \bigcup_{\Delta \in \mathcal Q_{C_0}(Q)} (1-c)\Delta.
\end{equation*}
Suppose that $\phi_1$, $\phi_2$ have Fourier supports in $\mathit \Sigma_1$ and $\mathit \Sigma_2$ respectively which satisfy the normalization \eqref{eqn:energy assum} and the relaxed margin condition \eqref{Bmargin}. Then, 
\begin{equation} \label{tte}
\| \phi_1 \phi_2 \|_{L^p(Q_R)} \le (1+Cc) \|[\Phi_1][\Phi_2] \|_{L^p(X(Q))} + c^{-C}, 
\end{equation}
where $Q$ is a cube of sidelength $CR$ contained in $C^2Q_R$ and \begin{equation} \label{tteMar}
\mathrm{marg} (\Phi_j) \ge \frac{1}{100} - \Big(\frac{2^{C_0}}{R} \Big)^{1/N}.
\end{equation}
\end{prop}

In the remaining parts of this section we will prove the above proposition. 
Consider the margin \eqref{tteMar}. From \eqref{Tmargin} and \eqref{bush function} it follows that
\begin{align*}
\mathrm{marg}(\Phi_j) &\ge \mathrm{marg}(\phi_j) - CR^{-1/2} \\
&\ge \frac{1}{100} - 2\Big( \frac{1}{R}\Big)^{1/N} - C\Big( \frac{1}{R}\Big)^{1/2} \\
&\ge \frac{1}{100} - \Big(\frac{2^{C_0}}{R} \Big)^{1/N}.
\end{align*}
The proof of  \eqref{tte} is accomplished through  many steps.
We begin with the following averaging lemma.
\begin{lem}[Lemma 6.1 in \cite{t1}] \label{averagelem}
  Let $R>0$, $0< c \le 2^{-C_0}$, and let $Q_R$ be a cube of sidelength $R$.
  If $f$ is a smooth function, then there exists a cube $Q$ of
  sidelength $CR$ contained in $C^2Q_R$ such that
  \[
  \| f \|_{L^p(Q_R)} \le (1+Cc) \| f \|_{L^p(X(Q))}.
  \]
\end{lem}
By this lemma,
\[
\| \phi_1\phi_2 \|_{L^p(Q_R)} \le (1+Cc)\| \phi_1\phi_2 \|_{L^p(X(Q))}.
\]
Using the triangle inequality we divide $\|\phi_1\phi_2\|_{L^p(X(Q))}$ into three parts:
\begin{align*}
\| \phi_1\phi_2 \|_{L^p(X(Q))} 
&\le \| [\Phi_1][\Phi_2] \|_{L^p(X(Q))}  \\
&\qquad + \| (\phi_1-[\Phi_1]) \phi_2 \|_{L^p(X(Q))}  + \| [\Phi_1](\phi_2-[\Phi_2]) \|_{L^p(X(Q))}.
\end{align*}
To prove \eqref{tte} it suffices to show
\begin{align}
\|(\phi_1-[\Phi_1]) \phi_2\|_{L^p(X(Q))}
&\lesssim c^{-C}, \label{estimate part} \\
\|[\Phi_1](\phi_2-[\Phi_2])\|_{L^p(X(Q))} &\lesssim c^{-C}.
\nonumber
\end{align}
Since these two estimates are similarly obtained, we will consider only the first one. 

\subsection{$L^1$-bilinear estimates}
\begin{lem}  \label{lem:L1} 
\[
\| (\phi_1-[\Phi_1]) \phi_2 \|_{L^1(X(Q))} \lesssim c^{-C}R .   
\]
\end{lem}
\begin{proof}
By the Cauchy--Schwarz inequality it suffices to show  
\begin{align}
  \| \phi_j \|_{L^2(Q)} &\lesssim R^{1/2}, \label{tr} \\
  \| \phi_j-[\Phi_j] \|_{L^2(Q)} &\lesssim c^{-C} R^{1/2},
  \label{mma} 
\end{align}
for \( j=1,2. \)
Consider \eqref{tr}. 
We have
\[
  \| \phi_j \|_{L^2(Q)}^2 \le
  \int_{- CR}^{CR} \| \phi_j(t) \|_{2}^2 dt
  \lesssim R E(\phi_j) \le R.
\]
Thus \eqref{tr} follows. 

Consider \eqref{mma}. By the triangle inequality and \eqref{tr} 
it suffices to show 
\begin{equation} \label{mma2}
  \|[\Phi_j]\|_{L^2(Q)} \lesssim c^{-C} R^{1/2}.
\end{equation}
We have that 
\[
   \| \Phi_j^{(\Delta)} \|_{L^2(\Delta)}^2  \le 
   \int_{t_{\Delta}-C2^{-C_0}R}^{t_{\Delta}+C2^{-C_0}R} 
   \| \Phi_j^{(\Delta)}(t) \|_{2}^2 dt 
   \lesssim 2^{-C_0}R  E(\Phi_j^{(\Delta)}).
\]
By using \eqref{wavepack},
\begin{align*}
\sum_{\Delta \in \mathcal Q_{C_0}(Q)} E(\Phi_j^{(\Delta)})  &=
\sum_{\Delta \in \mathcal Q_{C_0}(Q)} \Big\| \sum_{T_j \in \mathbf{T}_j} \frac{m_{\Delta,T_j}}{m_{T_j}}
\phi_{T_j}(0) \Big\|_2^2  \\
&\lesssim c^{-C}\sum_{\Delta \in \mathcal Q_{C_0}(Q)}  \sum_{T_j \in \mathbf{T}_j} \Big(\frac{m_{\Delta,T_j}}{m_{T_j}} \Big)^2 h_{T_j}^2.
\end{align*}
By \eqref{weighT} and \eqref{l2sum},
\begin{align*}
\sum_{\Delta \in \mathcal Q_{C_0}(Q)} E(\Phi_j^{(\Delta)})&\lesssim  c^{-C} \sum_{T_j \in \mathbf{T}_j}  h_{T_j}^2 \\
&\lesssim c^{-C} E(\phi_j).
\end{align*}
By the above estimates,
\[
  \|[\Phi_j]\|^2_{L^2(Q)} = \sum_{\Delta \in \mathcal Q_{C_0}(Q)} 
  \| \Phi_j^{(\Delta)} \|_{L^2(\Delta)}^2  \lesssim 
  2^{-C_0}c^{-C}R E(\phi_j) \le c^{-C}R.
\]
Thus we have \eqref{mma2}.
\end{proof}

\subsection{Orthogonality} \label{subsec:orth}

By interpolation it now suffices to show 
\begin{equation} \label{L2}
  \| (\phi_1-[\Phi_1]) \phi_2 \|_{L^2(X(Q))}
  \lesssim c^{-C} R^{-\frac{n-1}{4}}.
\end{equation}

By \eqref{bush decomp}, \eqref{core} and the triangle inequality, 
  \[
    \| (\phi_1-[\Phi_1]) \phi_2 \|_{L^2(X(Q))}
    \le  \sum_{\Delta \in \mathcal{Q}_{C_0}(Q)}
    \| \Phi_1^{(\Delta)} \phi_2 \|_{L^2(X(Q) \setminus \Delta)}.
  \]
Since the number of cubes in $\mathcal{Q}_{C_0}(Q)$ is
$2^{(n+1)C_0} \lesssim c^{-C}$, it is reduced to showing
\begin{equation} \label{WolffBush}
  \| \Phi_1^{(\Delta)}\phi_2 \|_{L^2(X(Q) \setminus \Delta)}
  \lesssim c^{-C} R^{-\frac{n-1}{4}}.
\end{equation}
Observe that if $q \in \mathcal Q_{\kappa}(Q)$ meets $X(Q)
\setminus \Delta$, then $\mathrm{dist}(q,\Delta) \ge cR $. 
By using  \eqref{stDec}, \eqref{wavepacket decomposition} and \eqref{bush function},
\begin{align} \label{eqn:AF}
\| \Phi_1^{(\Delta)}\phi_2 \|_{L^2(X(Q) \setminus \Delta)}^2
&\lesssim \sum_{\substack{
q \in \mathcal{Q}_{\kappa}(Q): \\
\mathrm{dist}(q,\Delta) \ge cR}}
\| \Phi_1^{(\Delta)}\phi_2 \|_{L^2(\psi_q^2)}^2 \nonumber \\
&= \sum_{\substack{
q \in \mathcal{Q}_{\kappa}(Q): \\
\mathrm{dist}(q,\Delta) \ge cR}}
\int_{\psi_q^4} \bigg| \sum_{T_1 \in \mathbf T_1}
\frac{m_{\Delta,T_1}}{m_{T_1}}\phi_{T_1}(z)
\sum_{T_2 \in \mathbf T_2} \phi_{T_2}(z) \bigg|^2 dz 
\end{align}
where $z$ denotes $(x,t)$.
We write the integral in the above equation as 
\begin{equation} \label{innerForm}
  \sum_{\substack{T_1,T_1' \in \mathbf T_1,\\
                    T_2,T_2' \in \mathbf T_2}}
  \int_{\psi_q^4} \bigg(
  \frac{m_{\Delta,T_1}}{m_{T_1}} \phi_{T_1}(z) \phi_{T_2}(z) \bigg)
  \bigg(\frac{m_{\Delta,T_1'}}{m_{T_1'}} \overline{\phi}_{T_1'}(z)
  \overline{\phi}_{T_2'}(z) \bigg) dz.
\end{equation}

We define $S$ to be the set of $(v_1,v_1',v_2,v_2') \in V_1 \times V_1 \times V_2 \times V_2$ such that 
\begin{equation} \label{geo condition}
\begin{aligned}
  v_1+v_2 &= v_1' + v_2'  + O(r^{-1}), \\
  |v_1|^2+|v_2|^2 &= |v_1'|^2+|v_2'|^2 + O(r^{-1}).
\end{aligned}
\end{equation}

\begin{lem} \label{lem:vanish_inn}
Let $S^c = V_1 \times V_1 \times V_2 \times V_2 
\setminus S$ and let $q$ be a cube of side-length 
$R^{1/2}$. Suppose that $T_1, T_1' \in \mathbf T_1$, $T_2, T_2' \in
\mathbf T_2$ satisfy $(v(T_1),v(T_1'),v(T_2),v(T_2')) \in
S^c$. Then,
\begin{equation} \label{vaniTerm}
  \bigg| \int_{\psi_q^4} 
  \phi_{T_1}(z) \phi_{T_2}(z) 
  \overline\phi_{T_1'}(z)
  \overline\phi_{T_2'}(z) dz  \bigg| =0.
\end{equation}
\end{lem}
\begin{proof}
By Parseval's formula the left side of \eqref{vaniTerm} equals to
\begin{equation*}
  \Big| \langle \widehat{
  \psi_q\phi_{T_1}} \ast \widehat{\psi_q\phi_{T_2}} ,~
  \widehat{\psi_q\phi_{T_1'}} \ast
  \widehat{\psi_q\phi_{T_2'}} \rangle \Big| ,
\end{equation*}
where the hat $\widehat{\quad}\, $ denotes the spacetime Fourier 
transform.

By the construction of wave packets $\phi_{T_j}$, we see that $\widehat{\psi_q\phi_{T_j}}$ is supported on a $O(r^{-1})$-neighborhood of $(v(T_j), -|v(T_j)|^2/2)$, and that
\( \widehat{\psi_q\phi_{T_1}} \ast \widehat{\psi_q\phi_{T_2}}\) 
is supported on a $O(r^{-1})$-neighborhood of
$\big(v(T_1)+v(T_2),~-(|v(T_1)|^2+|v(T_2)|^2)/2 \big)$. 
Thus we can see that if $(v(T_1),v(T_1'),v(T_2),v(T_2'))$ 
does not satisfy \eqref{geo condition} then the supports of $ \widehat{\psi_q\phi_{T_1}} \ast \widehat{\psi_q\phi_{T_2}}$ and $\widehat{\psi_q\phi_{T_1'}} \ast
\widehat{\psi_q\phi_{T_2'}}$ are disjoint, so we have \eqref{vaniTerm}.
\end{proof}
By the above lemma,
\begin{equation} \label{restrT}
(\ref{innerForm}) \le  
  \sum_{(T_1,T_1',T_2,T_2') \in \mathbf S}
  \bigg|\int_{\psi_q^4} \bigg(
  \frac{m_{\Delta,T_1}}{m_{T_1}} \phi_{T_1} \phi_{T_2} \bigg)
  \bigg(\frac{m_{\Delta,T_1'}}{m_{T_1'}} \overline\phi_{T_1'}
  \overline\phi_{T_2'} \bigg) dz\bigg| 
\end{equation} where
\begin{equation*}
  \mathbf S =\{(T_1,T_1',T_2,T_2') \in \mathbf T_1 \times \mathbf T_1 \times \mathbf T_2 \times \mathbf T_2:  
  (v(T_1),v(T_1'),v(T_2),v(T_2')) \in S \}.
\end{equation*}
Using the arithmetic-geometric mean inequality, we divide the integral of the above estimate as follows:
\begin{align*}
  & \bigg|\int_{\psi_q^4} \bigg(
  \frac{m_{\Delta,T_1}}{m_{T_1}} \phi_{T_1}(z) \phi_{T_2}(z) \bigg) \bigg(
  \frac{m_{\Delta,T_1'}}{m_{T_1'}} \overline\phi_{T_1'}(z)
  \overline\phi_{T_2'}(z) \bigg) dz \bigg|\\
  & \qquad \qquad \lesssim  \int_{\psi_q^2}  \bigg|
  \frac{ m_{\Delta,T_1'}}{m_{T_1}} \Psi_{T_1'}(z_q)  
  \frac{\phi_{T_1 }(z) \phi_{T_2}(z)}{\Psi_{ T_1}(z_q) }  
  \frac{ }{} \bigg|^2 dz\\
  & \qquad \qquad  \qquad \qquad +\int_{\psi_q^2}    \bigg|
  \frac{m_{\Delta,T_1}}  {m_{T_1'}}{\Psi_{ T_1}(z_q) }
  \frac{\overline \phi_{T_1'}(z) \overline \phi_{T_2'}(z)}{\Psi_{ T_1'}(z_q) }  \bigg|^2 dz,
\end{align*}
where $z_q$ is the center of $q$. The two integrals of the right side in the above equation are of the same form. By combining \eqref{eqn:AF}, \eqref{restrT} and the above estimate, the estimate \eqref{WolffBush} is reduced to showing
\begin{equation} \label{reduced main}
  \sum_{\substack{
  q \in \mathcal{Q}_{\kappa}(Q) :\\
  \mathrm{dist}(q,\Delta) \ge cR}}
  \sum_{(T_1,T_1',T_2,T_2') \in \mathbf S}
  \int_{\psi_q^2}
  \bigg|
  \frac{ m_{\Delta,T_1'}}{m_{T_1}} \Psi_{ T_1'}(z_q)  
  \frac{\phi_{T_1 }(z)\phi_{T_2}(z)}{\Psi_{ T_1}(z_q) }   \bigg|^2 dz 
  \lesssim  c^{-C}R^{-\frac{n-1}{2}}.
\end{equation}

We separate the summation $\sum_{(T_1,T_1',T_2,T_2') \in \mathbf S}$
into two parts
\[
  \sum_{(T_1,T_1',T_2,T_2') \in \mathbf S} = 
  \sum_{T_1 \in \mathbf T_1, T_2 \in \mathbf T_2}
  \sum_{\substack{
  T_1' \in \mathbf T_1, T_2' \in \mathbf T_2 :\\
  (v(T_1),v(T_1'),v(T_2),v(T_2')) \in  S}}.
\]
By rearranging the left side of \eqref{reduced main} is bounded by
\[
\sum_{\substack{q \in \mathcal{Q}_{\kappa}(Q) :\\
\mathrm{dist}(q,\Delta) \ge cR}}
\sum_{\substack{T_1 \in \mathbf T_1,\\
T_2 \in \mathbf T_2}} \Bigg( \frac{1}{m_{T_1} }
\int_{\psi_q^2}  \bigg|\frac{\phi_{T_1}(z)\phi_{T_2}(z)}{ \Psi_{ T_1}(z_q)} \bigg|^2 dz \hspace{-6pt}
 \sum_{\substack{
T_1' \in \mathbf T_1, T_2' \in \mathbf T_2 :\\
(v(T_1),v(T_1'),v(T_2),v(T_2')) \in S}} \hspace{-18pt}
{m_{\Delta,T_1'}} \Psi_{ T_1'}^2(z_q)   \Bigg),
\]
where ${  (\frac{m_{\Delta,T}}{m_T})^2 \le 
\frac{m_{\Delta,T}}{m_T}  }$ is used. 
To show \eqref{reduced main} it suffices to prove the following two estimates: 
\begin{equation} \label{geometric1}
\max_{\substack{q \in \mathcal{Q}_{\kappa}(Q) :\\
\mathrm{dist}(q,\Delta) \ge cR}}
\max_{\substack{T_1 \in \mathbf T_1,\\
T_2 \in \mathbf T_2}} 
   \sum_{\substack{
  T_1' \in \mathbf T_1, T_2' \in \mathbf T_2 :\\
  (v(T_1),v(T_1'),v(T_2),v(T_2')) \in  S}}
   {m_{\Delta,T_1'}}  \Psi_{ T_1'}^2(z_q) 
   \lesssim  c^{-C} R^{1/2}
\end{equation}
and
\begin{equation} \label{SS1}
  \sum_{q \in \mathcal{Q}_{\kappa}(Q)}
\sum_{\substack{T_1 \in \mathbf T_1,\\
T_2 \in \mathbf T_2}}  \frac{1}{m_{T_1} }
\int_{\psi_q^2}  \bigg|\frac{\phi_{T_1}(z) \phi_{T_2}(z)}{ \Psi_{ T_1}(z_q)} \bigg|^2 dz \\
  \lesssim c^{-C}R^{-n/2}.
\end{equation}

\subsection{Proof of the estimate {(\ref{geometric1})}}
We take a close look at the condition  
\[(v(T_1),v(T_1'),v(T_2),v(T_2')) \in S.\]
From the first equation of \eqref{geo condition} we have 
\begin{equation} \label{v2A}
  v_2' = v_1+v_2-v_1' + O(r^{-1}).
\end{equation}
Inserting this into the second equation of 
\eqref{geo condition}, we have
\[
   |v_1|^2 - |v_1+v_2-v_1'|^2 = |v_1'|^2
   -|v_2|^2+O(r^{-1}),
\]
which is equivalent to
\begin{equation*}
 (v_1'-v_1) \cdot (v_1'-v_2)=O(r^{-1}).
\end{equation*}
Let $\sigma(v_1,v_2) \subset \mathbb R^n$  be the sphere 
of radius $\frac{|v_1-v_2|}{2}$ with center 
$\frac{v_1+v_2}{2}$. Then
the above equation means that $v_1'$ lies in 
the $O(r^{-1})$-neighborhood $\sigma(v_1,v_2;Cr^{-1})$ of the sphere $\sigma(v_1,v_2)$. Thus, if $v_1,~ v_2$ are given then one has $v_1' \in \sigma(v_1,v_2;Cr^{-1})$. Also, if $v_1,~v_2,~,v_1'$ are given then by \eqref{v2A} we see that $v_2'$ is contained in a ball $B(v_1+v_2-v_1';Cr^{-1})$.

Now we use this observation. For given $T_1$, $T_2$, 
we have that $v(T_1')$ is contained in 
\[
\sigma(v(T_1),v(T_2);Cr^{-1}).
\] 
If $v(T_1), v(T_2), v(T_1')$ is determined then $v(T_2')$ 
has $O(1)$ choices, and if $v(T_2')$ is determined then the 
number of $T_2'$ passing through $z_q$ is only one. 
Thus to show \eqref{geometric1} it suffices to show that for given $T_1 \in \mathbf T_1$, $T_2 \in \mathbf T_2$ and $q \in \mathcal Q_{\kappa}(Q)$ with $\mathrm{dist}(q,\Delta) \ge cR$,
\begin{equation} \label{geometric est}
  \sum_{\substack{T_1' \in \mathbf T_1:\\v(T_1') \in V_1 
  \cap  \sigma(v_1,v_2;Cr^{-1})}}
  m_{\Delta,T_1'}  \Psi_{ T_1'}^2(z_q)
  \lesssim  c^{-C} R^{1/2},
\end{equation}
where $v_1=v(T_1)$ and $v_2=v(T_2)$.
By \eqref{eqn:weight}, we have
\begin{equation} \label{eqn:conTphi}
  \sum_{\substack{T_1' \in \mathbf T_1:\\v(T_1') \in V_1 \cap
  \sigma(v_1,v_2;Cr^{-1})}}
  m_{\Delta,T_1'}  \Psi_{ T_1'}^2(z_q)
  \lesssim \sum_{T_2 \in \mathbf T_2} \int_{4Q} |\phi_{T_2}(z)|^2 
  \Gamma_{q}(z) dz + R^{-C},
\end{equation}
where
\[
  \Gamma_q(z) :=
  \sum_{\substack{T_1' \in \mathbf T_1: \\
  v(T_1') \in V_1 \cap \sigma(v_1,v_2;Cr^{-1}) }}
  \psi_{\Delta}(z) \Psi_{ T_1'}^2(z) \Psi_{ T_1'}^2(z_q).
\]
Consider $\Gamma_{q}(z)$. The union of $T_1'$ passing through $z_q$ with $v(T_1') \in V_1 \cap \sigma(v_1,v_2;Cr^{-1})$ forms a conic set $\Lambda_{\pi_1}(z_q;Cr)$ where  
\begin{equation*} 
\pi_1 := {\sigma(v_1,v_2)\cap  \tilde\Xi_1}.
\end{equation*}
From $\mathrm{dist}(\Delta,q) \ge cR$ it follows that the tubes  $T_1'$ passing through $z_q$ can overlap at most $O(1)$ times on $\Delta$. Thus,
\begin{equation*} 
\int_{4Q} |\phi_{T_2}(z)|^2
\Gamma_{q}(z) dz \lesssim c^{-C}
\int_{4Q} |\phi_{T_2}(z)|^2 \bigg( 1+
  \frac{\mathrm{dist}(z,\Lambda_{\pi_1}(z_q))}{r}\bigg)^{-N^{10}} dz.
\end{equation*}
By combining this with \eqref{eqn:conTphi},
\[
\sum_{\substack{T_1' \in \mathbf T_1:\\v(T_1') \in V_1 \cap
\sigma(v_1,v_2;Cr^{-1})}}
m_{\Delta,T_1'}  \Psi_{ T_1'}^2(z_q) 
\lesssim c^{-C}\sum_{T_2 \in \mathbf T_2} \int_{4Q} |\phi_{T_2}(z)|^2 \bigg( 1+
\frac{\mathrm{dist}(z,\Lambda_{\pi_1}(z_q))}{r}\bigg)^{-N^{10}} dz + R^{-C}.
\] 
By a dyadic decomposition, to prove \eqref{geometric est} it suffices to show
\begin{equation} \label{trans_app}
  \sum_{T_2 \in \mathbf T_2}\|\phi_{T_2}\|^2_{L^2(4Q \cap \Lambda_{\pi_1}(z_q;Cr))} \lesssim  c^{-C} r.
\end{equation}

We observe that  for any $w_2 \in \Xi_2$ and $w_1,w_1' \in \pi_1$ with $w_1 \neq w_1'$, there exists $0 \le \varepsilon_0 < 1$ such that
\begin{equation} \label{transv}
\Big\langle  \frac{w_2-w_1}{|w_2-w_1|}, \frac{w_1'-w_1}{|w_1'-w_1|}
\Big\rangle \le \varepsilon_0 .
\end{equation}
Indeed, if we take $\xi_2 \in \sigma(v_1,v_2)$
such that $w_1'+\xi_2= v_1+v_2$, that is, $(w_1'+\xi_2)/2$ is the center
of $\sigma(v_1,v_2)$,  then 
\(
  \langle  w_1-\xi_2 , w_1'-w_1\rangle=0.
\)
Using this we have 
\[ 
\text{the left side of \eqref{transv}} = \Big\langle  \frac{w_2-\xi_2}{|w_2-w_1|}, \frac{w_1'-w_1}{|w_1'-w_1|}
\Big\rangle. 
\]
So the above equation is bounded by \(
  {|w_2-\xi_2|}/{|w_2-w_1|}.
\) From the definition of $\Xi_1$ and $\Xi_2$, we can see that there is $0\le \varepsilon_0 <1$ such that ${|w_2-\xi_2|}/{|w_2-w_1|} \le \varepsilon_0$. Thus we have \eqref{transv}.
By Lemma \ref{englem}, 
\begin{equation*} 
\text{the left side of \eqref{trans_app}} \lesssim  r \sum_{T_2 \in \mathbf T_2} E(\phi_{T_2}).  
\end{equation*}
By \eqref{wavepack}, \eqref{l2sum} and \eqref{eqn:energy assum},
\begin{equation} \label{roughSS}
\sum_{T_2 \in \mathbf T_2} E(\phi_{T_2}) \lesssim c^{-C}  \sum_{T_2 \in \mathbf T_2} h_{T_2}^2 \lesssim c^{-C}.
\end{equation}
By combining these two estimates we obtain \eqref{trans_app}. 


\qed

\subsection{Proof of the estimate (\ref{SS1})}
Consider the integral in the left side of \eqref{SS1}. 
By applying Lemma \ref{lem:simBil} to the left side of \eqref{SS1} it suffices to show
\begin{equation} \label{sumAllq} 
  \sum_{q \in \mathcal{Q}_{\kappa}(Q) }
\sum_{\substack{T_1 \in \mathbf T_1,\\
T_2 \in \mathbf T_2}}  \frac{1}{m_{T_1} }
\frac{\|\phi_{T_1}\|_{L^2({\psi_q})}^2}{\Psi_{ T_1}^2(z_q)} 
\|\phi_{T_2}\|_{L^2({\psi_q})}^2
  \lesssim c^{-C} R^{1/2}.
\end{equation} 
The left side of the above equation is written as
\[
  \sum_{ T_1 \in \mathbf T_1} \Bigg(
  \sup_{q \in \mathcal{Q}_{\kappa}(Q) }
  \frac{\|\phi_{T_1}\|_{L^2({\psi_q})}^2}{ \Psi_{ T_1}^4(z_q)} \Bigg)
\Bigg( \frac{1}{m_{T_1}}
  \sum_{q \in \mathcal{Q}_{\kappa}(Q) }
  \sum_{T_2 \in \mathbf T_2}
\Psi_{ T_1}^2(z_q)\|\phi_{T_2}\|_{L^2({\psi_q})}^2  \Bigg).
\]
Consider the inner summand.
Since both the width of $T$ and the sidelength of $q$ are $\sim r$, by some basic estimates it follows that
for any tube $T \in \mathbf T_1 \cup \mathbf T_2$ and $q \in \mathcal Q_\kappa$,
\begin{equation} \label{elemEst}
\Psi_{ T}(z)\psi_q^{1/2}(z) \sim \Psi_{ T}(z_q).
\end{equation}
%
%
%
Using this equation we have
\[
\sum_{q \in \mathcal{Q}_{\kappa}(Q) }
\sum_{T_2 \in \mathbf T_2}
\Psi_{ T_1}^2(z_q)\|\phi_{T_2}\|_{L^2({\psi_q})}^2 
\lesssim 
\sum_{q \in \mathcal{Q}_{\kappa}(Q) }
\sum_{T_2 \in \mathbf T_2}
\|\Psi_{ T_1}\phi_{T_2}\|_{L^2(\psi_q^{3/2})}^2,
\]
which is $\lesssim m_{T_1}$ by \eqref{WT}.
Now to prove \eqref{sumAllq} it is enough to show
\[
\sum_{ T_1 \in \mathbf T_1}
\sup_{q \in \mathcal{Q}_{\kappa}(Q) } \frac{
\|\phi_{T_1}\|_{L^2({\psi_q})}^2}{\Psi_{ T_1}^4(z_q)}  \lesssim c^{-C} R^{1/2}.
\]
By \eqref{wavepack} and \eqref{elemEst}, 
\[
\frac{\|\phi_{T_1}\|_{L^2({\psi_q})}^2}{\Psi_{ T_1}^4(z_q)} \lesssim  c^{-C}  r^{-n} h_{T_1}^2 \frac{\| \Psi_{ T_1}^2\|_{L^2(\psi_q)}^2}{\Psi_{ T_1}^{4}(z_q)}
\lesssim  c^{-C}  r h_{T_1}^2.
\]
Therefore, by \eqref{l2sum}
\[
\sum_{ T_1 \in \mathbf T_1}
\sup_{q \in \mathcal{Q}_{\kappa}(Q) } \frac{
\|\phi_{T_1}\|_{L^2({\psi_q})}^2}{\Psi_{ T_1}^4(z_q)}  
\lesssim c^{-C}  r \sum_{ T_1 \in \mathbf T_1} h_{T_1}^2 
\lesssim c^{-C} R^{1/2}.
\]\qed

\section{Proof of Proposition \ref{prop:sharp}}
\label{sec:nonPP}

By Proposition \ref{prop:tao} it suffices to show
\begin{equation} \label{sumEE}
\|[\Phi_1][\Phi_2]\|_{L^p(Q)} \le (1+Cc) \overline{\mathcal K}(R/C_0)E(\phi_1)^{1/2}E(\phi_2)^{1/2}.
\end{equation}
By \eqref{core},  
\begin{equation*} 
  \|[\Phi_1][\Phi_2]\|_{L^p(Q)}  \le \sum_{\Delta \in \mathcal Q_{C_0}(Q)}
  \|\Phi_1^{(\Delta)}\Phi_2^{(\Delta)} \|_{L^p(\Delta)}.
\end{equation*}
From Proposition \ref{prop:tao} we see that 
$\mathrm{marg}(\Phi_j^{(\Delta)})
\ge 1/100 - (2^{C_0}/R)^{1/N}$.
By Definition \ref{defn:hypoth} the right side of the above estimate is bounded by
\begin{equation*} 
\overline{\mathcal K}(R/C_0) \sum_{\Delta \in \mathcal Q_{C_0}(Q)} E(\Phi_1^{(\Delta)})^{1/2}E(\Phi_2^{(\Delta)})^{1/2}.
\end{equation*}
By the Cauchy--Schwarz inequality, this is bounded by
\begin{equation*} 
\overline{\mathcal K}(R/C_0)  \bigg( \sum_{\Delta \in \mathcal Q_{C_0}(Q)}
E(\Phi_1^{(\Delta)}) \bigg)^{1/2} \bigg( \sum_{\Delta \in \mathcal Q_{C_0}(Q)}
E(\Phi_2^{(\Delta)}) \bigg)^{1/2}.
\end{equation*}
Thus, to show \eqref{sumEE} it suffices to show 
\[
\bigg( \sum_{\Delta \in \mathcal Q_{C_0}(Q)}
E(\Phi_j^{(\Delta)}) \bigg)^{1/2}
\le (1+Cc) E(\phi_j)^{1/2}
\]
for $j=1,2$. 

\begin{lem} \label{tube_sum}
Let $\mathcal Q$ be a finite index set. Suppose that $m_{q,T_j}$ are non-negative numbers such that
\begin{equation} \label{weight}
  \sum_{q \in \mathcal Q} m_{q,T_j} \le 1
\end{equation}
for all $T_j \in \mathbf T_j$.
Then,
\begin{equation} \label{Bessel}
  \bigg(\sum_{q \in \mathcal Q} E \Big(\sum_{T_j \in \mathbf T_j}
  m_{q,T_j} \phi_{T_j} \Big) \bigg)^{1/2}\le (1+Cc)E(\phi_j)^{1/2}.
\end{equation}
\end{lem}
\begin{proof}
To get the constant $(1+Cc)$, we need to consider the Fourier support
of $f_{T_j}$. Let 
\[
  Y_j = \bigcup_{v_j \in V_j} \{ \xi \in B_{v_j} : \mathrm{dist}(\xi, \tilde \Xi_j \setminus B_{v_j}) > Cc^2 R^{-1/2} \}.
\]
We define the operator $P_{\Omega(Y)}$ by
\[
  \widehat{P_{\Omega(Y)}f} = \chi_{\Omega(Y)} \widehat f.
\]
By Minkowski's inequality, 
\begin{align*}
  \bigg(\sum_{q \in \mathcal Q} E \big(\sum_{T_j\in \mathbf T_j} m_{q,T_j}
  \phi_{T_j} \big) \bigg)^{1/2} &=
  \bigg(\sum_{q \in \mathcal Q} \Big\|\sum_{T_j\in \mathbf T_j} m_{q,T_j}
  \phi_{T_j}(0) \Big\|_2^2 \bigg)^{1/2} \\
  &=\bigg(\sum_{q \in \mathcal Q} \Big\|\sum_{T_j\in \mathbf T_j} m_{q,T_j}
  \int_G \eta^{x(T_j)} P_{\Omega,v(T_j)}
  \phi_j(0) d\Omega \Big\|_2^2 \bigg)^{1/2} \\
  & \le \int_G \bigg(\sum_{q \in \mathcal Q}
  \Big\|\sum_{T_j\in \mathbf T_j} m_{q,T_j}\eta^{x(T_j)} P_{\Omega,v(T_j)} \phi_j(0)  \Big\|_2^2 \bigg)^{1/2} {d\Omega},
\end{align*}
which is less than or equal to the sum of
\begin{equation}\label{B'v}
  \int \bigg(\sum_{q \in \mathcal Q}
  \Big\|\sum_{T_j\in \mathbf T_j} m_{q,T_j}\eta^{x(T_j)} P_{\Omega,v(T_j)}
  P_{\Omega(Y_j)}\phi_j(0)  \Big\|_2^2 \bigg)^{1/2} {d\Omega}
\end{equation}
and
\begin{equation}\label{Bv}
  \int \bigg(\sum_{q \in \mathcal Q}
  \Big\|\sum_{T_j\in \mathbf T_j} m_{q,T_j}\eta^{x(T_j)}  P_{\Omega,v(T_j)}
  (1-P_{\Omega(Y_j)})\phi_j(0)  \Big\|_2^2 \bigg)^{1/2} {d\Omega}.
\end{equation}
To prove \eqref{Bessel}, it suffices to show that
\[
  \eqref{B'v} \le E(\phi_j)^{1/2} \qquad \text{and} \qquad
  \eqref{Bv}  \le Cc E(\phi_j)^{1/2}.
\]
Consider \eqref{B'v}. From \eqref{spPar} we see that the Fourier transform of $\eta^{x_0}$ is supported in $D(0;c^{2}R^{-1/2})$. So, the Fourier support of
$\eta^{x_0} P_{\Omega,v_j}P_{\Omega(Y_j)} f$ is contained in $B_{v_j}$.
By orthogonality,
\eqref{B'v} is bounded by
\begin{equation*}
  \int \bigg(\sum_{q \in \mathcal Q}\sum_{v_j \in V_j}
  \Big\|\sum_{x_0 \in L} m_{q,T_j^{(x_0,v_j)}}
  \eta^{x_0}  P_{\Omega,v_j}
  P_{\Omega(Y_j)}\phi_j(0)  \Big\|_2^2 \bigg)^{1/2} {d\Omega}.
\end{equation*}
By rearranging, it is equal to
\[
\int \bigg( \sum_{v_j \in V_j} \int | P_{\Omega,v_j}
P_{\Omega(Y_j)}\phi_j(x,0)|^2 \sum_{q \in \mathcal Q}\bigg|\sum_{x_0 \in
L} {m_{q,T_j^{(x_0,v_j)}}} \eta^{x_0}(x) \bigg|^2 dx\bigg)^{1/2}
{d\Omega},
\]
which is bounded by
\[
\int \bigg( \sum_{v_j \in V_j} \int | P_{\Omega,v_j}
P_{\Omega(Y_j)}\phi_j(x,0)|^2 \bigg|\sum_{q \in \mathcal Q}\sum_{x_0 \in
L} {m_{q,T_j^{(x_0,v_j)}}} \eta^{x_0}(x) \bigg|^2 dx\bigg)^{1/2}
{d\Omega}.
\]
By \eqref{spaceDecom} and \eqref{weight}, this is bounded by
\[
  \int \bigg(\sum_{v_j \in V_j} \int | P_{\Omega,v_j}
P_{\Omega(Y_j)}\phi_j(x,0)|^2 dx\bigg)^{1/2} {d\Omega}.
\]
By orthogonality, the above is bounded by
\[
  \int \| P_{\Omega(Y_j)}\phi_j(0)\|_2 {d\Omega}.  
\]
Since $\| P_{\Omega(Y_j)}\phi_j(0)\|_2 \le E(\phi_j)^{1/2}$, we have that $\eqref{B'v} \le E(\phi_j)^{1/2}$.

Consider \eqref{Bv}. Apply the previous arguments but using
almost orthogonality instead of orthogonality. Then we have
\[
\eqref{Bv} \lesssim \int \| (1-P_{\Omega(Y_j)})\phi_j(0)\|_2 {d\Omega}.
\]
By the Cauchy--Schwarz inequality this is bounded by
\[
 \Big(\int \| (1-P_{\Omega(Y_j)})\phi_j(0)\|_2^2 {d\Omega} \Big)^{1/2}.
\]
By Plancherel's theorem and rearranging the integrals, this is equal
to
\[
 \Big(\int \| (1-\chi_{ \Omega(Y_j) })
 \widehat{\phi_j(0)}\|_2^2 {d\Omega} \Big)^{1/2}
 =  \Big(\int \Big( \int (1-\chi_{\Omega(Y_j)})(\xi)
 {d\Omega} \Big) |\widehat{\phi_j(0)}(\xi)|^2 d\xi \Big)^{1/2}.
\]
By a direct calculation we have that for any $\xi \in \tilde \Xi_j$,
\begin{align*}
 \int 1-\chi_{\Omega(Y_j)}(\xi) {d\Omega}
 &= \frac{1}{|D(0;CR^{-1/2})|}\int_{D(0;CR^{-1/2})}
 1-\chi_{Y_j}(\xi+w) dw\\
 &\lesssim c^2.
\end{align*}
Inserting this into the previous, we obtain that $\eqref{Bv}
\lesssim c^2 E(\phi_j)$.
\medskip

\end{proof}

\section{A localization operator}
In this section we introduce a localization operator and state some relevant basic estimates. When exploiting energy concentrations, the localization operator is used as a tool.

\smallskip

By \eqref{FForm} we have
\(
\widehat{\phi_j(t)}(\xi) = e^{-\pi i (t-t_0) |\xi|^2} \widehat{\phi_j(t_0)}(\xi),
\) which is written as
\begin{equation} \label{otExp}
\phi_j(t) = \mathcal U_j[\phi_j(t_0)](t-t_0).
\end{equation}
\begin{defn} \label{def:P_D}
  Let $D=D(x_D,t_D;r)$ be a disc. We define an operator $P_D \phi_j$ by
  \begin{equation} \label{def:PDt}
    P_D \phi_j(t) = \mathcal U_j [(\chi_D \ast
    \eta_{r^{1-1/N}}) \phi_j(t_D)](t-t_D)
  \end{equation}
  where $\eta_r$ is defined as \eqref{fourier_bump}.
\end{defn}

\begin{lem} \label{lem:ess supp}
Let $r \ge C_0$, $D = D(x_D,t_D;r)$ and 
\[
D^{\pm} := D(x_D,t_D;r(1 \pm r^{-1/2N})).
\] 
Suppose that $\phi_j$ 
satisfies that $\mathrm{marg}(\phi_j) \ge C_0 r^{-1+1/N}$ 
for $j=1,2$. Then, 
  \begin{equation} \label{Pmar}
    \mathrm{marg}(P_D \phi_j) \ge \mathrm{marg}(\phi_j) 
    - C_0r^{-1+1/N}
  \end{equation}
and 
  \begin{align}
    \| P_D \phi_j(t_D) \|_{L^2(\mathbb R^{n} \setminus D^+)} 
    &\lesssim r^{-N} E(\phi_j)^{1/2}, \label{PD1}\\
    \| (1-P_D) \phi_j \|_{L^2(D^-)} 
    &\lesssim r^{-N} E(\phi_j)^{1/2}, \label{PD2} \\
    E(P_D \phi_j) &\le \|\phi_j\|_{L^2(D^+)}^2
    + Cr^{-N}E(\phi_j), \label{PD3}\\
    E((1-P_D)\phi_j) &\le \|\phi_j(t_D)\|^2_{L^2(\mathbb R^n 
    \setminus  D^-)} + Cr^{-N}E(\phi_j) \label{PD4}.
  \end{align}
\end{lem}
\begin{proof}
  {Consider  \eqref{Pmar}. Observe that the size of $\supp \widehat{\phi}_j$ is comparable to that of $\supp \widehat{\phi_j(0)}$. Similarly, the size of $\supp \widehat{P_D \phi_j}$ and  $\supp \widehat{P_D \phi_j(t_D)}$ is comparable.
From \eqref{def:PDt} we have
\begin{equation} \label{P_Dsupp}
  \widehat{P_D \phi_j(t_D)} = (\widehat{\chi_D} \widehat{\eta_{r^{1-1/N}}}) \ast \widehat{\phi_j(t_D)}.
\end{equation}
Since $\widehat{\eta_{r^{1-1/N}}}(\xi) = \widehat{\eta}(r^{1-1/N} \xi)$ is supported on $D(0;r^{-1+1/N})$, the Fourier support of $P_D \phi_j(t_D)$ is expanded $O(r^{-1+1/N})$ more than that of $\phi_j(t_D)$.}  
Thus we have \eqref{Pmar}.

We have that
\(
\frac{1}{2}D \subset D^- \subset D \subset D^+ \subset 2D.
\)
From this relation it follows that
  \begin{gather}
    0 \le \chi_D \ast \eta_{r^{1-1/N}} \le 1 , \nonumber \\
    \chi_D \ast \eta_{r^{1-1/N}}(x) \lesssim r^{-N} \quad 
    \text{ for $x \in \mathbb R^{n} \setminus D^+ $}, 
    \label{Ob1} \\
    1-\chi_D \ast \eta_{r^{1-1/N}}(x) \lesssim r^{-N} \quad 
    \text{ for $x \in D^-$ }.  \label{Ob2} 
  \end{gather}
  Indeed, the first one is trivial. Consider \eqref{Ob1}.
  We have that
  \begin{equation} \label{ptD}
    \chi_D \ast \eta_{r^{1-1/N}}(x) \lesssim
    \bigg(1+ \frac{\mathrm{dist}(x,D)}{r^{1-1/N}} 
    \bigg)^{-M}, \qquad  \forall M>0.
  \end{equation}
  If $x \in \mathbb R^n \setminus D^+$ then
  \[
    \mathrm{dist}(x,D) \ge r(1+r^{-1/2N}) - r = r^{1-1/2N}.
  \]
  By inserting this into the previous inequality we can 
  obtain \eqref{Ob1}. 
  
  Consider \eqref{Ob2}. We have that 
  \begin{equation} \label{pt(1-D)}
    1- \chi_D \ast \eta_{r^{1-1/N}}(x) 
    \lesssim  \bigg(1+ \frac{\mathrm{dist}
    (x, \mathbb R^n \setminus D)}{r^{1-1/N}}  \bigg)^{-M},
    \qquad \forall M>0.
  \end{equation}
  If $x \in D^-$, then
  \[
    \mathrm{dist}(x,\mathbb R^n \setminus D) 
    \ge r - r(1-r^{-1/2N}) = r^{1-1/2N}.
  \]
 Thus we have \eqref{Ob2}.

 \medskip
  
  Now consider from \eqref{PD1} to \eqref{PD4}.
  By \eqref{def:PDt} and \eqref{Ob1}  it follows that
  \[
    \| P_D \phi_j (t_D)\|_{L^2(\mathbb R^n \setminus D^+)}
    = \| (\chi_D \ast \eta_{r^{1-1/N}}) \phi_j(t_D) 
    \|_{L^2(\mathbb R^n \setminus D^+)} 
    \lesssim r^{-N} E(\phi_j)^{1/2}.
  \]
  So we have \eqref{PD1}. Similar arguments also give \eqref{PD2}.

  From \eqref{PD1} it follows that
  \begin{align*}
    E(P_D \phi_j) &=  \| P_D \phi_j(t_D) \|_2^2 \\
    & \le \| P_D \phi_j \|^2_{L^2(D^+)} 
    + Cr^{-N} E(\phi_j)^{1/2}  \\
    & \le \| \phi_j \|^2_{L^2(D^+)} + Cr^{-N} 
    E(\phi_j)^{1/2} .
  \end{align*}
  So we have \eqref{PD3}. Analogously we have \eqref{PD4}.
\end{proof}

Now we consider some properties of $P_D \phi_j$ in $\mathbb R^n \times \mathbb R$.

\begin{lem} \label{lem:ptPD}
  Let $ r \ge C_0$ and  $D = D(x_D,t_D;r)$. Suppose that 
  $\phi_j$ satisfies that $\mathrm{marg}(\phi_j) \ge 
  C_0r^{-1+1/N}$. Then, 
\begin{align} \label{pointPD}
  |P_D \phi_j(x, t)&| \le C_M r^{n/2}
  \bigg( 1+ \frac{\mathrm{dist} \big((x,t),
  \Lambda_j(x_D,t_D;r) \big)}{r^{1-1/N}} \bigg)^{-M}
  E(\phi_j)^{1/2}, \qquad \forall M >0 \\
  \intertext{and \vspace{-8pt}}
\label{extPD}
  &\| (1-P_D) \phi_j \|_{L^\infty(Q(x_D,t_D;r/4 ))}
  \lesssim r^{-N}  E(\phi_j)^{1/2}.
\end{align}
\end{lem}
\begin{proof}
Consider \eqref{pointPD}.  By \eqref{def:PDt} and 
\eqref{kerF}, 
\begin{equation*}
  P_D \phi_j(x,t) = \int K_{j,t-t_D}(x-y) (\chi_D \ast
  \eta_{r^{1-1/N}})(y) \phi_j(y,t_D) dy.
\end{equation*}
By Lemma \ref{lem:ker_est} and \eqref{ptD}, 
\[
  | K_{j,t-t_D}(x-y) (\chi_D \ast 
  \eta_{r^{1-1/N}})^{1/2}(y) |
  \le C_M
  \bigg( 1+ \frac{\mathrm{dist} \big((x,t),
  \Lambda_j(x_D,t_D;r) \big)}{r^{1-1/N}} \bigg)^{-M}
\]
for any $M>0$.
Thus,
\begin{align*}
  |P_{D} \phi_j(x,t) | &\le C_M
  \bigg( 1+ \frac{\mathrm{dist} \big((x,t),
  \Lambda_j(x_D,t_D;r) \big)}{r^{1-1/N}} \bigg)^{-M}
  \int (\chi_D \ast \eta_{r^{1-1/N}})^{1/2}(y) 
  |\phi_j(y,t_D)| dy \\
  &\le C_M
  r^{n/2} \bigg( 1+ \frac{\mathrm{dist} \big((x,t),
  \Lambda_j(x_D,t_D;r) \big)}{r^{1-1/N}} \bigg)^{-M} 
  E(\phi_j)^{1/2},
\end{align*}
where the last line follows from the Cauchy--Schwarz 
inequality.
\smallskip

Consider \eqref{extPD}. Similarly, by \eqref{def:PDt} and 
\eqref{kerF}, it is written as
\begin{equation*}
  (1-P_D) \phi_j(x,t) = \int K_{j,t-t_D}(x-y) (1-\chi_D 
  \ast \eta_{r^{1-1/N}})(y) \phi_j(y,t_D) dy.
\end{equation*}
By the Cauchy--Schwarz inequality,
\begin{equation} \label{1-pd}
    |(1-P_D) \phi_j(x,t)| \le 
    \bigg( \int \big| K_{j,t-t_D}(x-y) (1-\chi_D \ast
    \eta_{r^{1-1/N}})(y) \big|^2 dy \bigg)^{1/2} 
    E(\phi_j)^{1/2}.
\end{equation}
By Lemma \ref{lem:ker_est} and \eqref{pt(1-D)} we have that
for $(x,t) \in Q(x_D,t_D;r/4)$,
\begin{align*}
 \int \big| K_{j,t-t_D}(x-y) (1-\chi_D \ast
 \eta_{r^{1-1/N}})(y) \big|^2 dy 
 &\lesssim r^{-M} \int | K_{j,t-t_D}(x-y)| dy \\
 &\lesssim r^{-M}
\end{align*}
for any $M>0$.
Substituting this in \eqref{1-pd} we can obtain 
\eqref{extPD}.

\end{proof}

\begin{lem} \label{lem:Epd}
Let $ r \ge C_0$ and  $D = D(x_D,t_D;r)$. 
Then, for each $t_0 \in \mathbb R$ there is a disc $\mathcal G_{j,t_0}(D)$ of radius $A_w|t_D-t_0|/2 + 4r$ in $\mathbb R^n \times \{t_0\}$ such that 
$\mathcal G_{j,t_0}(D)$ contains $\Lambda_j(x_D, t_D;r) \cap (\mathbb R^n \times \{t_0\})$ and 
\begin{equation*} 
E(P_D \phi_j) \le \|\phi_j\|_{L^2(\mathcal G_{j,t_0}(D))} + Cr^{-N}E(\phi_j).
\end{equation*}
\begin{proof}
By \eqref{otExp} and \eqref{def:PDt},
\begin{align}
P_D\phi_j(x,t_D) &= (\chi_D \ast
\eta_{r^{1-1/N}})(x)\mathcal U_j[\phi_j(t_0)](x,t_D-t_0) 
\nonumber \\
&= \int (\chi_D \ast \eta_{r^{1-1/N}})(x)K_{j,t_D-t_0}(x-y)  \phi_j(y,t_0) dy. \label{f1}
\end{align}
If we ignore Schwarz tails, the equation $K_{j,t_D-t_0}(x-y)$ implies that $x-y$ is contained in $\Lambda_{j,t_D-t_0}$ by Lemma \ref{lem:ker_est}. The $\Lambda_{j,t_D-t_0}$ is contained in a disc of radius $A_w|t_D-t_0|/2 + C$, so $x-y$ is contained in a disc of radius $A_w|t_D-t_0|/2+C$. Since $x$ is contained in $D$, we see that that $y$ is contained in a disc $D_j^\star:=D(x_j,t_0;A_w|t_D-t_0|/2+ 2r)$ for some $x_j$.   
Using the symmetric property of $\Lambda_j$ about the origin, we also have that $y-x$ is contained in $\Lambda_{j,t_0-t_D}$, which implies $(y,t_0) \in  \Lambda_j(x,t_D)$. So, we can see that $D_j^{\star}$ contains $\Lambda_j(x_D,t_D;r) \cap \mathbb (R^n \times \{t_0\})$.
Using this observation we have
\[
|(\chi_D \ast \eta_{r^{1-1/N}})^{1/2}(x)K_{j,t_D-t_0}(x-y)| \le C_{M} \bigg(1+ \frac{\mathrm{dist}(y,D_j^\star)}{r^{1-1/N}} \bigg)^{-M}, \quad \forall M>0.
\]
Let $\tilde D_j^{\star} := D(x_j,t_0;A_w|t_D-t_0|/2+2r(1+r^{-1/2N}))$.
 By the above estimate,
\[
|(\chi_D \ast \eta_{r^{1-1/N}})^{1/2}(x)K_{j,t_D-t_0}(x-y)(1-\chi_{\tilde D_j^{\star}}(y))| \le C_M r^{-M}, \quad \forall M>0.
\]
We divide $\phi_j(t_0) = \chi_{\tilde D_j^{\star}} \phi_j(t_0) + (1-\chi_{\tilde D_j^{\star}}) \phi_j(t_0)$ and insert it into \eqref{f1}. Then,
\begin{align*}
\|P_D\phi_j(t_D) \|_{2} & \le \|\mathcal U_j[\chi_{\tilde D_j^{\star}} \phi_j(t_0)](t_D-t_0) \|_{2} +Cr^{-N}E(\phi_j)^{1/2} \\
&\le \|\chi_{\tilde D_j^{\star}} \phi_j(t_0)\|_{2} +Cr^{-N}E(\phi_j)^{1/2} \\
&\le \|\phi_j\|_{L^2(D(x_j,t_0;A_w|t_D-t_0|/2+4r))} +Cr^{-N}E(\phi_j)^{1/2}.
\end{align*}
If we take $\mathcal G_{j,t_0}(D) := D(x_j,t_0;A_w|t_D-t_0|/2+4r)$ then we have the desired estimate. Since $D_j^{\star}$ contains $\Lambda(x_D,t_D;r) \cap \mathbb (R^n \times \{t_0\})$, the $\mathcal G_{j,t_0}(D)$ also contains  $\Lambda_j(x_D, t_D;r) \cap (\mathbb R^n \times \{t_0\})$.
\end{proof}

\end{lem}

\section{Proof of Proposition \ref{prop:energy concent}} 
\label{conPP}
\noindent Let $\phi_1$, $\phi_2$ satisfy the margin  \eqref{mar_cond} and the normalization 
\eqref{eqn:energy assum}. Let  $Q=Q(x_Q,t_Q;R)$ be a cube of sidelength $R$ and centered at $(x_Q,t_Q)$ and let $I_Q=[t_Q-R/2, t_Q+R/2]$ be the time interval of $Q$.

By Definition \ref{defn:hypoth} it suffices to show that for each $0 < \varepsilon \ll 1$
\begin{equation*} 
\| \phi_1 \phi_2\|_{L^p(Q)} \le (1-C_0^{-C})
\sup_{\substack{r \ge C_0^{-C}R,  \\ r/100< \mathring r \le (A_wA_d^{-1}+C_0^{-C})r}}  \mathcal K_{\varepsilon}(R, r, \mathring r ).
\end{equation*}
By Definition \ref{defn:hypoth2} we have
\[
  \| \phi_1 \phi_2 \|_{L^p(Q)} \le \mathcal K_{\varepsilon}(R,r,\mathring r) E_{r,\mathring r, t_e}^{\varepsilon}(\phi_1,\phi_2)^{1/p'}.
\]
It suffices to show the following proposition.

\begin{prop} \label{lem:energy}
Let $\phi_1$, $\phi_2$ be the same as described above and $R \ge 2^{C_0}$.  If $0< \delta \le C_0^{-C}$, then
for $0 < \varepsilon \ll 1$, there exist $t_e \in I_Q$, $r \ge C_0^{-C} R$ and $r/100< \mathring r \le (2A_wA_d^{-1} + C_0^{-C})r$ such that
\[
E^{\varepsilon}_{r, \mathring r,t_e}(\phi_1, \phi_2) \le 1-\delta.
\]
\end{prop}
\begin{proof}
Let $t \in I_Q$ and let $\mathfrak D_t^\varepsilon(\phi_1,\phi_2)$ and $\mathcal N(D)$ be defined as in Definition \ref{def:EC}. We define $\mathbf D^{\varepsilon}_{t}(\delta)$ to be the collection of discs $D \in \mathfrak D_t^\varepsilon(\phi_1,\phi_2)$ such that there exist $D_1, D_2 \subset \mathcal N(D)$ of radius $(2A_wA_d^{-1} + C_0^{-C})r_D$ satisfying
\begin{equation*} \label{eqn:concent}
\|\phi_1\|_{L^2(D_1)}  \|\phi_2\|_{L^2(D_2)} \ge 1- \delta.
\end{equation*}
For a disc $D$ we define $\mathring r_D$ to be the infimum of the radii of the discs $D_1,D_2 \subset \mathcal N(D)$ satisfying the above inequality, that is,
\[
\mathring r_D := \inf \{r: \|\phi_1\|_{L^2(D_1)}  \|\phi_2\|_{L^2(D_2)} \ge 1- \delta \text{ for }  D_1,D_2 \subset \mathcal N(D) \text{ with } r_{D_1}=r_{D_2}=r \}.
\]
Then for $D \in \mathbf D_t^{\varepsilon}(\delta)$,
\begin{equation} \label{boundofring}
\mathring r_D \le (2A_wA_d^{-1} + C_0^{-C})r_D.
\end{equation}
Since $\mathring r_D$ is the infimum, we have 
\begin{equation*} 
\sup_{\substack{D_1,D_2 \subset \mathcal N(D)\\: r_{D_1}=r_{D_2} \le \mathring r_D}}\|\phi_1\|_{L^2(D_1)}  \|\phi_2\|_{L^2(D_2)} \le 1- \delta.
\end{equation*}
In fact, since $\phi_1$ and $\phi_2$ are smooth, we have the equality
\begin{equation} \label{RingRadi}
\sup_{\substack{D_1,D_2 \subset \mathcal N(D)\\: r_{D_1}=r_{D_2} \le \mathring r_D}}\|\phi_1\|_{L^2(D_1)}  \|\phi_2\|_{L^2(D_2)} = 1- \delta.
\end{equation}
Let
\begin{equation} \label{radiusECen}
r^\delta(t) :=  \inf_{D \in \mathbf D_{t}^{\varepsilon} (\delta)}  r_D.
\end{equation}
Since $\phi_1$ and $\phi_2$ are smooth, there is a disc $D \in \mathbf D_t^{\varepsilon}(\delta)$ of radius $r^{\delta}(t)$. Let $\mathring r^{\delta}(t)$ is the infimum of radii $\mathring r_D$ for $D \in \mathbf D_t^{\varepsilon}(\delta)$ with radius $r^{\delta}(t)$, i.e.,
\begin{equation} \label{infRing}
\mathring r^{\delta}(t) := \inf_{D \in \mathbf D_t^{\varepsilon}(\delta) : r_D = r^{\delta}(t)} \mathring r_D.
\end{equation}
Then from \eqref{boundofring} it follows that for each $t \in I_Q$,
\begin{equation} \label{smallBigcomp}
\mathring r^{\delta}(t) \le (2A_wA_d^{-1} + C_0^{-C})r^{\delta}(t).
\end{equation}

To show that $\mathring r^\delta(t) \ge  r^\delta(t)/100$, 
let $D \in \mathbf D_t^{\varepsilon}(\delta)$ be a disc of radius $r^{\delta}(t)$ with $\mathring r_D = \mathring r^{
\delta}(t)$, and let $D_1,\, D_2 \subset \mathcal N(D)$ be the discs of radius $\mathring r^{\delta}(t)$. If we suppose $\mathring r^\delta(t) < r^\delta(t)/100$, then by \eqref{radiusECen} the distance between $D_1$ and $D_2$ is larger than $C^{-1}r^{\delta}(t)$.
Since the Fourier transform of $\phi_j(t)$ is compactly supported, for any proper disc $D \subset \mathbb R^n \times \{t\}$ we have $\|\phi_j\|_{L^2(D)} \gneq 0$.  
Thus, there is a disc $D'$ of radius $\lneq r^\delta(t)$ such that \[
\|\phi_1\|_{L^2(\mathcal N(D')}  \|\phi_2\|_{L^2(\mathcal N(D'))} = 1- \delta.
\]
This implies $D' \in \mathbf D_t^{\varepsilon}(\delta)$ but it contradicts $\eqref{radiusECen}$. Thus we have $\mathring r^\delta(t) \ge  r^\delta(t)/100$.

We also have that for each $t \in I_Q$,
\begin{equation} \label{EnergyCons}
  E^{\varepsilon}_{r^{\delta}(t), \mathring r^{\delta}(t), t} (\phi_1,\phi_2) = 1- \delta.
\end{equation}
Indeed, let $D$ be a  disc of radius $r^{\delta}(t)$ in $\mathbb R^n \times \{t\}$. If $D \in \mathbf D_t^{\varepsilon}(\delta)$ then by \eqref{RingRadi} and \eqref{infRing},
\[
\sup_{\substack{D_1, D_2 \subset \mathcal N(D)\\:r_{D_1}=r_{D_2}=\mathring r^{\delta}(t)}}\|\phi_1\|_{L^2(D_1)}  \|\phi_2\|_{L^2(D_2)} = 1-\delta.
\]
If $D \notin \mathbf D_t^{\varepsilon}(\delta)$ then from the definition of $\mathbf D_t^{\varepsilon}(\delta)$ it follows that for any discs $D_1, D_2 \subset \mathcal N(D)$ of radius $(2A_wA_d^{-1} + C_0^{-C})r^{\delta}(t)$, 
\[
\|\phi_1\|_{L^2(D_1)}  \|\phi_2\|_{L^2(D_2)} < 1-\delta.
\] 
Thus we have \eqref{EnergyCons}.

We choose a time $t_e \in I_Q$ such that
\begin{equation} \label{max radius}
  \frac{1}{2}\sup_{t \in I_Q} r^\delta(t) \le r^\delta(t_e)
  \le \sup_{t \in I_Q} r^\delta(t).
\end{equation}
By \eqref{smallBigcomp} and \eqref{EnergyCons}, to prove the proposition it suffices to show that if $0< \delta \le C_0^{-C}$ then
\begin{equation} \label{max r rage}
  r^\delta(t_e) \ge C_0^{-C}R.
\end{equation}

Since $\phi_1$ and $\phi_2$ are smooth, from \eqref{EnergyCons} it follows that for each $t \in I_Q$, there exist discs $D_t^{\delta} \subset \mathbb R^n \times \{t\}$ of radius $r^{\delta}(t)$ and $D_1, D_2 \subset \mathcal N(D_t^{\delta})$ of radius $\mathring r^{\delta}(t)$ such that
\begin{equation} \label{eqn:minimumenergy}
\|\phi_1\|_{L^2(D_1)}  \|\phi_2\|_{L^2(D_2)} = 1- \delta.
\end{equation}

Let $(x_e,t_e)$ be the center of $D_{t_e}^{\delta}$,
\begin{equation*} 
  r_e = r^\delta(t_e)+C_0
\end{equation*}
and
\[
\Lambda_{j,e} := \Lambda_j(x_e,t_e; C_0^2 A_*r_e).
\]
To have \eqref{max r rage} it is enough to show
\begin{equation} \label{EneCone}
  \bigcup_{t \in I_Q} D_t^\delta \subset \bigcap_{j=1,2}
  C\Lambda_{j,e}.
\end{equation}
Indeed, since $\Lambda_{1,e}$ and $\Lambda_{2,e}$ meet transversely, the union $\bigcup_{t \in  I_Q} D_t^\delta$ is contained in a ball $B(x_e,t_e;C_0^Cr_e)$. By comparing the length of $I_Q$ with the radius $C_0^Cr_e$, we have $r_e \ge C_0^{-C}  R$, which implies \eqref{max r rage} because  $R \ge 2^{C_0}$.

\medskip

To show \eqref{EneCone}, by \eqref{max radius} it suffices to prove that
for each $t \in I_Q$, the disc $\mathcal N(D_t^\delta)$  intersects
both $\Lambda_{1,e}$ and $\Lambda_{2,e}$.
Suppose for contradiction that there exists $\mathcal N(D_t^\delta)$ contained in $\mathbb R^{n+1} \setminus
\Lambda_{j,e}$ for some $j=1,2$. Let
\begin{equation*}
D_e := D(x_e,t_e;C_0A_*r_e).
\end{equation*}
We decompose
\begin{equation} \label{esD}
\|\phi_j\|_{L^2(\mathcal N(D_t^\delta))}^2 \le
4\|P_{D_e}\phi_j\|_{L^2(\mathcal N(D^\delta_t))}^2 + 4\|(1-P_{D_e})\phi_j\|_{L^2(\mathcal N(D^\delta_t))}^2.
\end{equation}
From $\mathcal N(D_t^\delta) \subset \mathbb R^{n+1} \setminus \Lambda_{j,e}$, one has $\mathrm{dist}(\mathcal N(D_t^\delta), \Lambda_j(x_e,t_e;C_0A_*r_e)) \ge C_0 A_*r_e$.
Since $r_e \ge C_0$ and the radius of $\mathcal N(D_t^\delta)$ is $\le A_*r_e$, by \eqref{pointPD} 
\begin{align}
\|P_{D_e}\phi_j\|_{L^2(\mathcal N(D_t^\delta))} 
&\lesssim r_e^{n/2} (C_0 A_*r_e^{1/N})^{-N^{10}} E(\phi_j)^{1/2} |\mathcal N(D_t^\delta)|^{1/2} \nonumber \\
&\lesssim C_0^{-C}. \label{Din}
\end{align}
From \eqref{PD4} and $r_e \ge C_0$ it follows that
\begin{align}
\|(1-P_{D_e}) \phi_j \|^2_{L^2(\mathcal N(D_t^\delta))}
& \le E((1-P_{D_e}) \phi_j)  \nonumber \\
& \le \| \phi_j(t_e) \|_{L^2(\mathbb R^{n} \setminus D_e^-)}^2 + C_0^{-C}. \label{OutD}
\end{align}
Since $\mathcal N(D_{t_e}^\delta) \subset \frac{1}{2}D_{e} \subset D_{e}^{-}$, we have 
\[
\| \phi_j (t_e)\|_{L^2(\mathbb R^{n} \setminus D_e^-)}^2 \le 1-\| \phi_j \|_{L^2(D_e^-)}^2 \le 1-\| \phi_j \|_{L^2(\mathcal N(D_{t_e}^\delta))}^2.
\]
By \eqref{eqn:energy assum} and \eqref{eqn:minimumenergy},
\begin{equation} \label{lowerbb}
  \|\phi_j\|_{L^2(\mathcal N(D_t^{\delta}))} \ge  1-\delta.
\end{equation}
Combining the above two estimates we have
\[
\| \phi_j \|_{L^2(\mathbb R^{n} \setminus D_e^-)}^2 \le 2\delta.
\]
Substituting this in \eqref{OutD} we have
\[
\|(1-P_{D_e})\phi_j\|_{L^2(\mathcal N(D_t^\delta))}^2
\le 2\delta+ C_0^{-C}.
\]
By inserting this and \eqref{Din} into \eqref{esD}, it follows that
\[
\|\phi_j\|_{L^2(\mathcal N(D_t^\delta))}^2 \le 8\delta + 8C_0^{-C}.
\]
Comparing this with \eqref{lowerbb} we have
$1-8C_0^{-C} \le 10\delta$. However, since $0< \delta \le C_0^{-C}$ is small, this is a contradiction. 
\end{proof}

\section{Proof of Proposition \ref{prop:core}} \label{sec:core}
\noindent 
Let $\phi_1, \phi_2$ satisfy the margin requirement \eqref{mar_cond} and the energy normalization \eqref{eqn:energy assum}.
Let $Q$ be a cube of sidelength $R$.
By Definition \ref{defn:hypoth2} we may assume that $\phi_1$, $\phi_2$ satisfy
\begin{equation} \label{exCon}
\|\phi_1\phi_2\|_{L^p(Q)} = \mathcal K_{\varepsilon}(R,r,\mathring r)E_{r,\mathring r, t_e}^{\varepsilon}
(\phi_1,\phi_2)^{1/p'}.
\end{equation}
It suffices to show 
\begin{equation} \label{pf2}
\| \phi_1 \phi_2 \|_{L^p(Q)} \le (1+Cc)E_{r,\mathring r,t_e}^{\varepsilon}(\phi_1,
\phi_2)^{1/p'} \overline{\mathcal K}(R) + 2^{CC_0}.
\end{equation}
We may assume that 
\[
E_{r,\mathring r,t_e}^{\varepsilon}(\phi_1,\phi_2) = \sup_{\substack{D \in \mathfrak D_{t}^{\varepsilon}(\phi_1,\phi_2)\\: r_D = r}} \sup_{\substack{\mathring D_1, \mathring D_2 \subset \mathcal N(D)\\:r_{\mathring D_1}=r_{\mathring D_2}=\mathring r}}  \big( \|\phi_1\|_{L^2(\mathring D_1)}  \|\phi_2\|_{L^2(\mathring D_2)} \big).
\]
Indeed, if $E_{r,\mathring r,t_e}^{\varepsilon}(\phi_1,\phi_2)=1/2$, we can take $\tilde r \ge r$ and $\mathring r \le \mathring {\tilde r} \le (2A_wA_d^{-1}+C_0^{-C}) \tilde r$ such that 
\[
1/2=\sup_{\substack{D \in \mathfrak D_{t_e}^{\varepsilon}(\phi_1,\phi_2)\\: r_D = \tilde r}} \sup_{\substack{\mathring D_1, \mathring D_2 \subset \mathcal N(D)\\:r_{\mathring D_1}=r_{\mathring D_2}=\mathring{\tilde r}}} \big( \|\phi_1\|_{L^2(\mathring D_1)}  \|\phi_2\|_{L^2(\mathring D_2)} \big).
\]
Thus, in \eqref{pf2} we can replace $E_{r,\mathring r,t_e}^{\varepsilon}(\phi_1,\phi_2)$ with $E_{\tilde r,\mathring{\tilde r},t_e}^{\varepsilon}(\phi_1,\phi_2)$.

By the smoothness of $\phi_1$, $\phi_2$, there exists a disc  $D_e \in \mathfrak D_{t_e}^{\varepsilon}(\phi_1,\phi_2)$ of radius $r$ such that
\begin{equation} \label{DeTake}
\sup_{\substack{\mathring D_1, \mathring D_2 \subset \mathcal N(D_e)\\:r_{\mathring D_1}=r_{\mathring D_2}=\mathring r}} \big(  \|\phi_1\|_{L^2(\mathring D_1)}  \|\phi_2\|_{L^2(\mathring D_2)} \big)
=
E_{r,\mathring r, t_e}^{\varepsilon}(\phi_1,\phi_2).
\end{equation}
Set
\[
\Lambda_j^{[e]} := \Lambda_j(x_e,t_e;r'/2)
\] for $j=1,2$ where $(x_e,t_e)$ is the center of $D_e$ and $r':=\frac{\mathring r}{A_*}(1+(\frac{\mathring r}{A_*})^{-1/2N})$.

\subsection{}

\begin{figure}[htbp]
\begin{center}

\tikzset{every picture/.style={line width=0.75pt}} 

\begin{tikzpicture}[x=0.75pt,y=0.75pt,yscale=-1,xscale=1]

\draw   (260.37,204.18) .. controls (260.37,195.63) and (280.76,188.71) .. (305.9,188.71) .. controls (331.04,188.71) and (351.43,195.63) .. (351.43,204.18) .. controls (351.43,212.72) and (331.04,219.65) .. (305.9,219.65) .. controls (280.76,219.65) and (260.37,212.72) .. (260.37,204.18) -- cycle ;
\draw    (131.17,65.19) -- (261.24,207.44) ;
\draw    (260.59,42.9) -- (349.69,199.85) ;
\draw    (480.2,65.66) -- (350.13,207.91) ;
\draw    (350.78,43.37) -- (261.67,200.32) ;
\draw  [draw opacity=0] (314.55,138.74) .. controls (311.92,139.99) and (308.93,140.69) .. (305.76,140.69) .. controls (302.63,140.69) and (299.67,140) .. (297.07,138.79) -- (305.76,123.43) -- cycle ; \draw   (314.55,138.74) .. controls (311.92,139.99) and (308.93,140.69) .. (305.76,140.69) .. controls (302.63,140.69) and (299.67,140) .. (297.07,138.79) ;
\draw   (299.13,141.14) -- (297.39,138.93) -- (300.33,138.41) ;
\draw   (311.65,138.11) -- (314.28,139.12) -- (312.02,141.07) ;
\draw  [dash pattern={on 4.5pt off 4.5pt}]  (260.5,227.2) -- (351.5,227.2) ;
\draw [shift={(351.5,227.2)}, rotate = 180] [color={rgb, 255:red, 0; green, 0; blue, 0 }  ][line width=0.75]    (0,5.59) -- (0,-5.59)(10.93,-3.29) .. controls (6.95,-1.4) and (3.31,-0.3) .. (0,0) .. controls (3.31,0.3) and (6.95,1.4) .. (10.93,3.29)   ;
\draw [shift={(260.5,227.2)}, rotate = 0] [color={rgb, 255:red, 0; green, 0; blue, 0 }  ][line width=0.75]    (0,5.59) -- (0,-5.59)(10.93,-3.29) .. controls (6.95,-1.4) and (3.31,-0.3) .. (0,0) .. controls (3.31,0.3) and (6.95,1.4) .. (10.93,3.29)   ;

\draw (301,229.1) node [anchor=north west][inner sep=0.75pt]    {$r'$};
\draw (295.14,141.06) node [anchor=north west][inner sep=0.75pt]    {$A_{d}$};
\draw (408.94,85.95) node [anchor=north west][inner sep=0.75pt]    {$\Lambda _{2}^{[e]}$};
\draw (192.99,85.13) node [anchor=north west][inner sep=0.75pt]    {$\Lambda _{1}^{[e]}$};

\end{tikzpicture}

\caption{Intersection of $\Lambda_1$ and $\Lambda_2$}
\label{fig:ToE}
\end{center}
\end{figure}
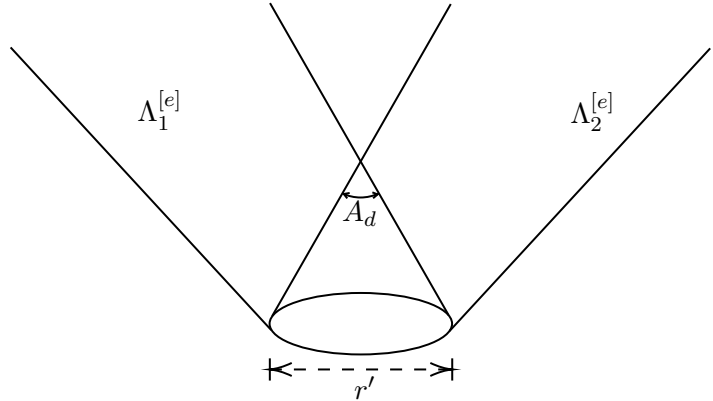

Consider the case that $Q$ intersects  both $\Lambda_1^{[e]}$ and $\Lambda_2^{[e]}$. Let $D_Q=D(x_Q,t_Q;4R)$ be the disc of radius $4R$ with the same center as $Q$. 
We decompose 
\begin{equation*} 
\begin{split}
\| \phi_1 \phi_2 \|_{L^p(Q)} \le &\|P_{D_Q} \phi_1 P_{D_Q}\phi_2\|_{L^p(Q)} \\
& \qquad +\| P_{D_Q} \phi_1 (1-P_{D_Q})\phi_2 \|_{L^p(Q)}
+\| (1-P_{D_Q})\phi_1 \phi_2 \|_{L^p(Q)}.
\end{split}
\end{equation*}
From \eqref{extPD} and H\"older's inequality it follows that
\begin{align*} 
  \| (1-P_{D_Q})\phi_1 \phi_2 \|_{L^p(Q)}  &\lesssim    
  R^{-N+C}, \\
  \| P_{D_Q}\phi_1 (1-P_{D_Q})\phi_2 \|_{L^p(Q)}
  &\lesssim R^{-N+C}.
\end{align*}
Both $P_{D_Q}\phi_1$ and $P_{D_Q}\phi_2$ satisfy the relaxed margin condition \eqref{Bmargin}, so by Proposition \ref{prop:sharp},
\begin{align*}
  \|P_{D_Q}\phi_1 P_{D_Q}\phi_2 \|_{L^p(Q)} &\le
  (1+Cc)E(P_{D_Q}\phi_1)^{1/2} E(P_{D_Q}\phi_2)^{1/2}\overline{\mathcal K}(R) 
  + 2^{CC_0} \\
  &\le
  (1+Cc) (E(P_{D_Q}\phi_1)^{1/2} E(P_{D_Q}\phi_2)^{1/2})^{1/p'} \overline{\mathcal K}(R) 
  + 2^{CC_0}.
\end{align*}
%
%
%
%
%
%
%
%
By \eqref{triK}, to prove \eqref{pf2} it suffices to show 
\begin{equation*}
E(P_{D_Q}\phi_1)^{1/2} E(P_{D_Q}\phi_2)^{1/2} \le E^{\varepsilon}_{r,\mathring r,t_e}(\phi_1,\phi_2)+ CR^{-N}.
\end{equation*}
By Lemma \ref{lem:Epd} there are  discs $\mathcal G_{1,t_e}(D_{Q})$, $\mathcal G_{2,t_e}(D_{Q})$ of radius 
\[
R_{e}:=A_w|t_Q-t_e|/2+ 16R
\] at time $t_e$ such that  the $\mathcal G_{j,t_e}(D_{Q})$ contains $\Lambda_j(x_Q,t_Q;4R) \cap (\mathbb R^n \times \{t_e\})$ and 
\[
E(P_{D_Q}\phi_j)^{1/2}  \le \|\phi_j\|_{L^2(\mathcal G_{j,t_e}(D_{Q}))} + CR^{-N}, \qquad j=1,2.
\]

To show 
\begin{equation} \label{leEn}
\|\phi_j\|_{L^2(\mathcal G_{1,t_e}(D_{Q}))}\|\phi_j\|_{L^2(\mathcal G_{2,t_e}(D_{Q}))} \le E^{\varepsilon}_{r,\mathring r,t_e}(\phi_1,\phi_2),
\end{equation}
we consider a geometric property of $\Lambda_1^{[e]} \cap \Lambda_2^{[e]}$. Since $\Lambda_1^{[e]} \cap \Lambda_2^{[e]}$ is a conic set (see Figure \ref{fig:ToE}), we can see that $A_d|t_Q-t_e| \le 2r' \le 4\mathring r/A_*$ and so
\[
R_{e} \le 2A_w A_d^{-1} \mathring r/A_* + 16R
\le (2A_w A_d^{-1} + 2000C_0^{-C})\mathring r/A_* \le \mathring r.
\]
Since $Q$ intersects $\Lambda_j^{[e]}$ and the $\mathcal G_{j,t_e}(D_{Q})$ contains $\Lambda_j(x_Q,t_Q;4R) \cap (\mathbb R^n \times \{t_e\})$, the $\mathcal G_{j,t_e}(D_{Q})$ intersects $\Lambda_{j}^{[e]} \cap (\mathbb R^n \times \{t_e\})$. To show that $\mathcal G_{j,t_e}(D_{Q})$ is contained in $\mathcal N(D_e)$ it suffices to show $r'+2R_e \le A_*r$.
Using $\mathring r/A_* \le r$ we have
\[
r'+2R_e \le 2\mathring r/A_*+ 2(2A_w A_d^{-1} + 2000C_0^{-C})\mathring r/A_* \le A_*r.
\]
Therefore we have \eqref{leEn}.

\subsection{}

Consider the case that $Q$ is contained in $\mathbb R^{n+1} \setminus \Lambda_j^{[e]}$ for some $j=1,2$. We only consider the case that $Q$ is contained in $\mathbb R^{n+1} \setminus \Lambda_1^{[e]}$, because the other case is similar. Let $\check D_{e} := D(x_{e},t_e;\frac{r}{400A_*})$. By the triangle inequality, 
\begin{equation} \label{IE1}
\| \phi_1 \phi_2 \|_{L^p(Q)} \le \|P_{\check D_{e}}\phi_1 \phi_2\|_{L^p(Q)}   + \| (1-P_{\check D_{e}})\phi_1 \phi_2 \|_{L^p(Q)}.
\end{equation}
Since $Q$ is contained in $\mathbb R^{n+1} \setminus \Lambda_j^{[e]}$ and $\frac{r'}{2} \ge \frac{r}{200A_*}$, from Lemma \ref{lem:ptPD} it follows that 
\begin{equation} \label{IE1_1}
\|P_{\check D_e}\phi_1 \phi_2 \|_{L^p(Q)} \le Cr^{-N^4}.
\end{equation}
By Definition \ref{defn:hypoth2},
\begin{align*}
\|(1-P_{\check D_e})\phi_1\phi_2\|_{L^p(Q)} &\le \mathcal K(R,r) E((1-P_{\check D_e})\phi_1)^{1/2p} E_{r,\mathring r,t_e}((1-P_{\check D_e})\phi_1,\phi_2)^{1/p'}  \\
&\le \mathcal K(R,r)E((1-P_{\check D_e})\phi_1)^{1/2p} E_{r,\mathring r,t_e}(\phi_1,\phi_2)^{1/p'}. 
\end{align*}
By \eqref{exCon}, 
\[
\|(1-P_{\check D_e})\phi_1\phi_2\|_{L^p(Q)}
\le E((1-P_{\check D_e})\phi_1)^{1/2p} \|\phi_1\phi_2\|_{L^p(Q)}.
\]
By \eqref{PD4},
\[
E((1-P_{\check D_e})\phi_1) \le 1- \|\phi(t_{e})\|^2_{L^2(\check D_e/2)} +Cr^{-N}.
\]
By \eqref{DetCen} and \eqref{eqn:energy assum},
\[
\|\phi(t_{e})\|_{L^2(\check D_e/2)} \ge \varepsilon
\] 
By combining the above three inequalities,
\[
\|(1-P_{\check D_e})\phi_1\phi_2\|_{L^p(Q)}
\le \kappa \|\phi_1\phi_2\|_{L^p(Q)}
\]
where $\kappa := (1- \varepsilon^2 +Cr^{-N} )^{1/2p}$.\\
By applying the triangle inequality to the right side of the above inequality,
\begin{align*}
\|(1-P_{\check D_e})\phi_1\phi_2\|_{L^p(Q)}
&\le \kappa \|P_{\check D_e}\phi_1\phi_2\|_{L^p(Q)} + \kappa \|(1-P_{\check D_e})\phi_1\phi_2\|_{L^p(Q)}.
\end{align*}
By rearranging,
\[
\|(1-P_{\check D_e})\phi_1\phi_2\|_{L^p(Q)} \le \frac{\kappa}{1-\kappa} \|P_{\check D_e}\phi_1\phi_2\|_{L^p(Q)}.
\]
By inserting this estimate into \eqref{IE1},
\[
\| \phi_1\phi_2\|_{L^p(Q)} \le \frac{1}{1-\kappa} \|P_{\check D_e}\phi_1\phi_2\|_{L^p(Q)}.
\]
From $\varepsilon \ge R^{-N/4}$ and $r \ge C_0^{C}R$, we have
\(
\frac{1}{1-\kappa} \lesssim  R^{CN},
\)
and
\[
\| \phi_1\phi_2\|_{L^p(Q)}  \lesssim R^{CN} \|P_{\check D_e}\phi_1\phi_2\|_{L^p(Q)}.
\]
By \eqref{IE1_1} we thus have \eqref{pf2}.

\section{Proof of Proposition \ref{prop:scale}.} \label{sec:endPP}
Suppose that $\phi_1$, $\phi_2$ obey the margin condition \eqref{mar_cond}. We may assume the normalization \eqref{eqn:energy assum}. It suffices to show that 
\begin{equation*}
  \|\phi_1 \phi_2\|_{L^p(Q_R)}
  \le (1+Cc) \mathcal K_{\varepsilon}(R/C_0, r_{\natural},\mathring r_{\natural}) E_{r,\mathring r,t_e}(\phi_1,\phi_2)^{1/p'}
  + c^{-C}.
\end{equation*}
We apply Proposition \ref{prop:tao}. Then it suffices to show
\begin{equation}  \label{recurGoal}
  \|[\Phi_1][\Phi_2]\|_{L^p(Q)}
  \le (1+Cc) \mathcal K_{\varepsilon}(R/C_0, r_{\natural},\mathring r_{\natural})E_{r,\mathring r, t_e}(\phi_1,\phi_2)^{1/p'} 
  + c^{-C},
\end{equation}
where the cube $Q$ is of side-length $CR$.
By \eqref{core}, 
\[
  \|[\Phi_1][\Phi_2]\|_{L^p(Q)} =
  \Big( \sum_{\Delta \in \mathcal Q_{C_0}(Q)}
  \|\Phi_1^{(\Delta)}\Phi_2^{(\Delta)} \|_{L^p(\Delta)}^p \Big)^{1/p}.
\]
Since the sidelength of $\Delta$ is $2^{-C_0}CR$, there is a cube of sidelength $R/C_0$ containing $\Delta$. So, by the margin of $\Phi_j$ shown in Proposition \ref{prop:tao} and Definition \ref{defn:hypoth2},
\[
  \|\Phi_1^{(\Delta)}\Phi_2^{(\Delta)} \|_{L^p(\Delta)} \le 
  \mathcal K_{\varepsilon}(R/C_0, r_{\natural},\mathring r_{\natural})
  E_{r_{\natural}, \mathring r_{\natural},t_e}(\Phi_1^{(\Delta)},\Phi_2^{(\Delta)})^{1/p'}
  \big( E(\Phi_1^{(\Delta)})^{1/2}E(\Phi_2^{(\Delta)})^{1/2} \big)^{1/p}.
\] 
By combining the above two equations,
\begin{align*}
&\|[\Phi_1][\Phi_2]\|_{L^p(Q)} \\
&\le  \mathcal K_{\varepsilon}(R/C_0,r_{\natural},\mathring r_{\natural})
\bigg( \sum_{\Delta \in \mathcal Q_{C_0}(Q)} 
E_{r_{\natural}, \mathring r_{\natural},t_e}(\Phi_1^{(\Delta)},\Phi_2^{(\Delta)})^{p/p'}
E(\Phi_1^{(\Delta)})^{1/2}E(\Phi_2^{(\Delta)})^{1/2} \bigg)^{1/p} \\
&\le  \mathcal K_{\varepsilon}(R/C_0,r_{\natural},\mathring r_{\natural}) \sup_{\Delta \in \mathcal Q_{C_0}(Q)} E_{r_{\natural}, \mathring r_{\natural},t_e}(\Phi_1^{(\Delta)},\Phi_2^{(\Delta)})^{1/p'}
\bigg( \sum_{\Delta \in \mathcal Q_{C_0}(Q)} 
E(\Phi_1^{(\Delta)})^{1/2}E(\Phi_2^{(\Delta)})^{1/2} \bigg)^{1/p}.
\end{align*}
By the Cauchy--Schwarz inequality and Lemma \ref{tube_sum}, it is bounded by
\begin{equation*}
(1+ Cc) \mathcal K_{\varepsilon}(R/C_0,r_{\natural},\mathring r_{\natural}) 
\sup_{\Delta \in \mathcal Q_{C_0(Q)}} 
E_{r_{\natural}, \mathring r_{\natural},t_e}(\Phi_1^{(\Delta)},\Phi_2^{(\Delta)})^{1/p'}.   
\end{equation*}
Now it suffices to show
\begin{equation*} 
\sup_{\Delta \in \mathcal Q_{C_0}(Q)} E_{r_{\natural}, \mathring r_{\natural},t_e}(\Phi_1^{(\Delta)},\Phi_2^{(\Delta)}) 
\le (1+Cc)E_{r,\mathring r, t_e}(\phi_1,\phi_2) + Cr^{-N},
\end{equation*}
because $\mathcal K(R,r,\mathring r) \lesssim \mathcal K(R) \lesssim R^C$.
By Lemma \ref{tube_sum} we have 
\(
E(\Phi_j^{\Delta})^{1/2} \le (1+Cc)E(\phi_j)^{1/2}
\).
Thus it is enough to show 
\[
  \|\Phi_j^{(\Delta)} \|_{L^2(D(z_0;r_{\natural}))}
  \le (1+Cc) \|\phi_j\|_{L^2(D(z_0;r))} + Cr^{-N}
\]
for all $\Delta \in \mathcal Q_{C_0}(Q)$ and all $z_0 \in \mathbb R^{n+1}$.
We have this estimate from the following lemma:
\begin{lem}
Let $\mathcal Q$ be a finite index set. Suppose that $m_{q,T_j}$ are non-negative numbers with \eqref{weight}.
Then, for any $r \ge 2C_0$ and $z_0 \in \mathbb R^{n+1}$,
\begin{align} \label{persistant}
  & \bigg( \sum_{q \in \mathcal Q} \big\| \sum_{T_j \in \mathbf T_j}
  m_{q,T_j} \phi_{T_j} \big\|^2_{L^2(D(z_0;r(1-Cr^{-1/3N})))}
  \bigg)^{1/2} \nonumber \\
  &\qquad \qquad \qquad \qquad \le (1+Cc)\|\phi_j\|_{L^2(D(z_0;r))} 
  + r^{-N}E(\phi_j)^{1/2}.
\end{align}
\end{lem}
\begin{proof}
Let $\tilde D=D(x_0,t_0;r(1-2r^{-1/3N}))$, $\tilde D'=D(x_0,t_0;r(1-r^{-1/2N}))$ and $\tilde D''=D(x_0,t_0;r)$. Then
we have $\tilde D \subsetneq \tilde D'\subsetneq \tilde D''$.\\
We divide
\[
  \sum_{T_j \in \mathbf T_j} m_{q,T_j} \phi_{T_j}
  = \sum_{T_j \in \mathbf T_j} m_{q,T_j} P_{\tilde D'}\phi_{T_j}
  + \sum_{T_j \in \mathbf T_j} m_{q,T_j} (1-P_{\tilde D'})\phi_{T_j}.
\]
Consider the first summation in the right side. We have
\[
  \sum_{q \in \mathcal Q} \Big\|  \sum_{T_j \in \mathbf T_j} m_{q,T_j}
  P_{\tilde D'}\phi_{T_j} \Big\|_{L^2(\tilde D)}^2
  \le \sum_{q \in \mathcal Q} E\Big( \sum_{T_j \in \mathbf T_j}
  m_{q,T_j}  P_{\tilde D'}\phi_{T_j} \Big).
\]
Applying Lemma \ref{tube_sum} we have
\[
  \sum_{q \in \mathcal Q} \Big\|  \sum_{T_j \in \mathbf T_j} m_{q,T_j}
  P_{\tilde D'}\phi_{T_j} \Big\|_{L^2(\tilde D)}^2 \le 
  (1+Cc)E(P_{\tilde D'} \phi_j).
\]
From \eqref{PD3} we can see that
\begin{align*}
  E(P_{\tilde D'}\phi_j) &\le \|\phi_j\|^2_{L^2(\tilde D'^+)}
  + Cr^{-N}E(\phi_j) \\
  &\le \|\phi_j\|^2_{L^2(\tilde D'')}
  + Cr^{-N}E(\phi_j),
\end{align*}
where $\tilde D' \subset \tilde D'^+ \subset \tilde D''$ is used.
Thus we obtain
\[
  \bigg(\sum_{q \in \mathcal Q} \Big\|  \sum_{T_j \in \mathbf T_j}
  m_{q,T_j}  P_{\tilde D'}\phi_{T_j} \Big\|_{L^2(\tilde D)}^2 \bigg)^{1/2}
  \le (1+Cc)\|\phi_j\|_{L^2(\tilde D'')} + Cr^{-N}E(\phi_j)^{1/2}.
\]
To prove \eqref{persistant} it now suffices to show
\begin{equation} \label{ErEs}
  \bigg(\sum_{q \in \mathcal Q} \Big\|  \sum_{T_j \in \mathbf T_j}
  m_{q,T_j} (1-P_{\tilde D'})\phi_{T_j} \Big\|_{L^2(\tilde D)}^2 \bigg)^{1/2}\lesssim
  r^{-N}E(\phi_j)^{1/2}.
\end{equation}
Since the operator $P_D$ is linear, we have
\[
  \sum_{T_j \in \mathbf T_j} m_{q,T_j}
  (1-P_{\tilde D'})\phi_{T_j} 
  = 
  (1-P_{\tilde D'}) \Big(\sum_{T_j \in \mathbf T_j} m_{q,T_j}\phi_{T_j} \Big).
\]
By $\tilde D  \subset \tilde D'^- \subset \tilde D'$ and \eqref{PD2},
\begin{align*}
  \Big\| 
  \sum_{T_j \in \mathbf T_j} m_{q,T_j}
    (1-P_{\tilde D'})\phi_{T_j}  \Big\|_{L^2(\tilde D)}
  &\le \Big\| 
  (1-P_{\tilde D'}) \Big(\sum_{T_j \in \mathbf T_j} m_{q,T_j}\phi_{T_j} \Big) \Big\|_{L^2(\tilde D'^-)} \\
  &\lesssim r^{-N} E \Big(\sum_{T_j \in \mathbf T_j} m_{q,T_j}\phi_{T_j} \Big)^{1/2}.
\end{align*} 
Using this we have
\[
  \bigg(\sum_{q \in \mathcal Q} \Big\|  \sum_{T_j \in \mathbf T_j}
  m_{q,T_j} (1-P_{\tilde D'})\phi_{T_j} \Big\|_{L^2(\tilde D)}^2 \bigg)^{1/2}
  \lesssim r^{-N} \bigg( \sum_{q \in \mathcal Q  } E \Big(\sum_{T_j \in \mathbf T_j} m_{q,T_j}\phi_{T_j} \Big) \bigg)^{1/2}.
\]
Thus, by Lemma \ref{tube_sum} we obtain \eqref{ErEs}.
\end{proof}

\end{document}